\documentclass[12pt,twoside]{amsart}
\usepackage{amsmath}
\usepackage{amsthm}
\usepackage{amsfonts}
\usepackage{amssymb}
\usepackage{latexsym}
\usepackage[all]{xy}

\date{}
\allowdisplaybreaks[4] \footskip=15pt
\renewcommand{\uppercasenonmath}[1]{}

\numberwithin{equation}{section} \theoremstyle{plain}
\newtheorem*{thm*}{Main Theorem}
\newtheorem{thm}{Theorem}[section]
\newtheorem{cor}[thm]{Corollary}
\newtheorem*{cor*}{Corollary}
\newtheorem{lem}[thm]{Lemma}
\newtheorem*{lem*}{Lemma}
\newtheorem{prop}[thm]{Proposition}
\newtheorem*{prop*}{Proposition}
\newtheorem{rem}[thm]{Remark}
\newtheorem*{rem*}{Remark}
\newtheorem{exa}[thm]{Example}
\newtheorem*{exa*}{Example}
\newtheorem{df}[thm]{Definition}
\newtheorem*{df*}{Definition}

\newtheorem*{conj*}{Conjecture}
\newtheorem{Fa}[thm]{Fact}
\newtheorem*{Fa*}{Fact}
\newtheorem{Qu}[thm]{Question}
\newtheorem*{Qu*}{Question}
\newtheorem*{ack*}{ACKNOWLEDGEMENTS}

\newcommand{\pf}{\noindent\begin {proof}}
\newcommand{\epf}{\end{proof}}

\newcommand{\Ext}{\mbox{\rm Ext}}
\newcommand{\Hom}{\mbox{\rm Hom}}
\newcommand{\Tor}{\mbox{\rm Tor}}

\def\Im{\mathop{\rm Im}\nolimits}
\def\Ker{\mathop{\rm Ker}\nolimits}
\def\Coker{\mathop{\rm Coker}\nolimits}

\def\cTr{\mathop{\rm cTr}\nolimits}
\def\coOmega{\mathop{{\rm co}\Omega}\nolimits}
\def\mod{\mathop{\rm mod}\nolimits}
\def\Mod{\mathop{\rm Mod}\nolimits}

\def\fd{\mathop{\rm fd}\nolimits}
\def\id{\mathop{\rm id}\nolimits}
\def\pd{\mathop{\rm pd}\nolimits}

\def\min{\mathop{\rm min}\nolimits}
\def\sup{\mathop{\rm sup}\nolimits}
\def\inf{\mathop{\rm inf}\nolimits}
\def\add{\mathop{\rm add}\nolimits}
\def\Add{\mathop{\rm Add}\nolimits}
\def\Prod{\mathop{\rm Prod}\nolimits}

\def\dim{\mathop{\rm dim}\nolimits}

\def\Hom{\mathop{\rm Hom}\nolimits}
\def\Ext{\mathop{\rm Ext}\nolimits}
\def\sup{\mathop{\rm sup}\nolimits}
\def\lim{\mathop{\underrightarrow{\rm lim}}\nolimits}
\def\limit{\mathop{\rm lim}\nolimits}

\def\grade{\mathop{\rm grade}\nolimits}

\def\Adst{\mathop{\rm Adst}\nolimits}
\def\Stat{\mathop{\rm Stat}\nolimits}

\def\amp{\mathop{\rm amp}\nolimits}
\def\height{\mathop{\rm ht}\nolimits}
\def\rad{\mathop{\rm rad}\nolimits}
\def\Spec{\mathop{\rm Spec}\nolimits}
\def\Ecograde{\mathop{\rm {E{\text-}cograde}}\nolimits}
\def\Tcograde{\mathop{\rm {T{\text-}cograde}}\nolimits}
\def\sEcograde{\mathop{\rm {s.E{\text-}cograde}}\nolimits}
\def\sTcograde{\mathop{\rm {s.T{\text-}cograde}}\nolimits}
\def\sgrade{\mathop{\rm s.grade}\nolimits}
\def\rad{\mathop{\rm rad}\nolimits}

\begin{document}
\begin{center}
{\large  \bf   Homological Aspects of the Dual Auslander Transpose II}

\vspace{0.5cm}
Xi Tang\\
%\bigskip
{\tiny{\it College of Science, Guilin University of Technology, Guilin 541004, Guangxi Province, P.R. China\\
E-mail: tx5259@sina.com.cn\\}}

\bigskip

Zhaoyong Huang\\
%\bigskip
{\tiny{\it Department of Mathematics, Nanjing University, Nanjing 210093, Jiangsu Province, P.R. China\\
E-mail: huangzy@nju.edu.cn\\}}
\end{center}
%\begin{figure}[b]
%\rule[-2.5truemm]{5cm}{0.1truemm}\\[2mm]
%{\small }
%\end{figure}

%\begin{figure}[b]
%\rule[-2.5truemm]{5cm}{0.1truemm}\\[2mm]
%{\small }
%\end{figure}
\bigskip
\centerline { \bf  Abstract}
\bigskip
\leftskip10truemm \rightskip10truemm \noindent
Let $R$ and $S$ be rings and $_R\omega_S$ a semidualizing bimodule. We prove that there exists a Morita equivalence between
the class of $\infty$-$\omega$-cotorsionfree modules and a subclass of the class of $\omega$-adstatic modules.
Also we establish the relation between the relative homological dimensions of a module $M$
and the corresponding standard homological dimensions of $\Hom(\omega,M)$. By investigating the properties of the Bass
injective dimension of modules (resp. complexes), we get some equivalent characterizations of semi-tilting modules
(resp. Gorenstein artin algebras). Finally we obtain a dual version of the Auslander-Bridger's approximation theorem.
As a consequence, we get some equivalent characterizations of Auslander $n$-Gorenstein artin algebras.
\vbox to 0.3cm{}

\section { \bf Introduction}

Semidualizing bimodules arise naturally in the investigation of various duality theories in commutative algebra. The study of such modules
was initiated by Foxby in \cite{Fox} and by Golod in \cite{Go}. Then Holm and White extended in \cite {HW} this notion to arbitrary associative rings,
while Christensen in \cite{Lar2} and Kubik in \cite{Ku} extended it to semidualizing complexes and quasidualizing modules respectively.
The study of semidualizing bimodules or complexes was connected to the so-called Auslander classes and Bass classes defined by Avramov and Foxby in \cite{Av}
and by Christensen in \cite{Lar2}. Semidualizing bimodules or complexes and the corresponding
Auslander/Bass classes have been studied by many authors; see, for example, \cite{ATY,Av,Lar2,Lar1,LH,EH,En,HW,TH} and so on.
In order to dualize the important and useful notions of Auslander transpose of modules and $n$-torsionfree modules,
we introduced in \cite{TH} the notions of cotranspose of modules and $n$-cotorsionfree modules with respect to
a semidualizing bimodule, and obtained several dual counterparts of interesting results.
Based on these mentioned above, we study further homological properties of cotranspose of modules, $n$-cotorsionfree modules
and related modules.

The paper is organized as follows.

In Section 2, we give some terminology and some preliminary results. In particular, we prove that if $(R,\mathfrak{m},k)$
is a commutative Gorenstein complete local ring with $\dim R>0$ and $\mathfrak{q}$ is a prime ideal of $R$ with
height non-zero, then the tensor product of the injective envelopes of $R/\mathfrak{q}$ and $k$ is equal to zero.
This gives a negative answer to an open question of Kubik posed in \cite{Ku} about quasidualizing modules.

Let $R$ and $S$ be rings and $_R\omega_S$ a semidualizing bimodule. In Section 3, we prove that if the projective dimension of $_R\omega$ is finite,
then the class of $\infty$-$\omega$-cotorsionfree modules is contained in the right orthogonal class of $_R\omega$; dually, if the projective dimension of
$\omega_S$ is finite, then the above inclusion relation between these two classes of modules is reverse.
Also we prove that there exists a Morita equivalence between
the class of $\infty$-$\omega$-cotorsionfree modules and a subclass of the class of $\omega$-adstatic modules. Finally,
we establish the relation between the relative homological dimensions of a module $M$
and the corresponding standard homological dimensions of $\Hom(\omega,M)$.

In Section 4, we first give some criteria for computing the Bass injective dimension of modules in term of the vanishing of
$\Ext$-functors and some special approximations of modules. Then, motivated by the philosophy of \cite{K},
we introduce the notion of semi-tilting bimodules in general case,
and prove that $_R\omega_S$ is right semi-tilting if and only if the Bass injective dimension of $_RR$ is finite.

In Section 5, we extend the Bass class and the Bass injective dimension of modules with respect to $\omega$ to that of homologically bounded complexes.
We show that a homologically bounded complex has finite Bass injective dimension if and only if it admits a special quasi-isomorphism in
the derived category of the category of modules. As an application of this result, we get some equivalent characterizations of
Gorenstein artin algebras.

In Section 6, we first introduce the notions of (strong) $\Ext$-cograde and $\Tor$-cograde of modules with respect to $\omega$.
Then we obtain a dual version of the Auslander-Bridger's approximation theorem
(\cite[Proposition 3.8]{F}) as follows. For any left $R$-module $M$ and $n\geqslant 1$, if the $\Tor$-cograde of $\Ext^i_R(\omega,M)$ with respect
to $\omega$ is at least $i$ for any $1\leqslant i\leqslant n$, then there exists a left $R$-module $U$ and a homomorphism $f: U\rightarrow M$
of left $R$-modules satisfying the following properties:
(1) The injective dimension of $U$ relative to the class of $\omega$-projective modules is at most $n$,
and (2) $\Ext^i_R(\omega,f)$ is bijective for any $1\leqslant i\leqslant n$. As an application of this result, we prove that
for any $n\geqslant 1$, the strong $\Ext$-cograde of $\Tor_i^S(\omega,N)$ with respect to $\omega$ is at least $i$ for any left $S$-module $N$
and $1\leqslant i \leqslant n$ if and only if the strong $\Tor$-cograde of $\Ext^i_R(\omega,M)$ with respect to $\omega$
is at least $i$ for any left $R$-module $M$ and $1\leqslant i \leqslant n$. Furthermore,
we get some equivalent characterizations of Auslander $n$-Gorenstein artin algebras.

\section {\bf Preliminaries}

Throughout this paper, $R$ and $S$ are
fixed associative rings with unites. We use $\Mod R$ (resp.
$\Mod S^{op}$) to denote the class of left $R$-modules (resp. right
$S$-modules). Let $M\in \Mod R$. We use $\pd_RM$, $\fd_RM$ and $\id_RM$ to denote the projective, flat
and injective dimensions of $M$ respectively, and use $\Add_RM$ (resp. $\Prod_RM$) to denote the subclass of $\Mod R$
consisting of all direct summands of direct sums (resp. direct products) of copies of $M$. We use
$$0\rightarrow M \rightarrow I^0(M) \stackrel{f^0}{\longrightarrow} I^1(M)\stackrel{f^1}{\longrightarrow}\cdots
\stackrel{f^{i-1}}{\longrightarrow} I^i(M)\stackrel{f^{i}}{\longrightarrow} \cdots\eqno{(1.1)}$$
to denote a minimal injective resolution of $M$. For any $n\geqslant 1$,
$\coOmega^n(M):= \Im f^{n-1}$ is called the {\bf $n$-th cosyzygy}
of $M$, and in particular, $\coOmega^{0}(M):=M$.

\begin{df} \label{df: 2.1} {\rm (\cite{HW}). (1) An ($R$-$S$)-bimodule $_R\omega_S$ is called
\textbf{semidualizing}\footnote{In \cite{TH} and the original version of this paper, we use $C$ to denote the given semidualizing module.
The referee suggests the following: ``The notation $\cTr_C M$ (see Definition 2.5 below) is very confusing. I am not sure how the first `c' 
is distinguished with the semidualizing module $C$, particularly when writing it on the
blackboard. It would be better to change the notation or quit using $C$ for the
semidualzing module." Following this suggestion, we denote the given semidualizing module by substituting $\omega$ for $C$.}
if the following conditions are satisfied:
\begin{enumerate}
\item[(a1)] $_R\omega$ admits a degreewise finite $R$-projective resolution.
\item[(a2)] $\omega_S$ admits a degreewise finite $S$-projective resolution.
\item[(b1)] The homothety map $_RR_R\stackrel{_R\gamma}{\rightarrow} \Hom_{S^{op}}(\omega,\omega)$ is an isomorphism.
\item[(b2)] The homothety map $_SS_S\stackrel{\gamma_S}{\rightarrow} \Hom_{R}(\omega,\omega)$ is an isomorphism.
\item[(c1)] $\Ext_{R}^{\geqslant 1}(\omega,\omega)=0.$
\item[(c2)] $\Ext_{S^{op}}^{\geqslant 1}(\omega,\omega)=0.$
\end{enumerate}

(2) A semidualizing bimodule $_R\omega_S$ is called {\bf faithful} if the following conditions are satisfied:
\begin{enumerate}
\item[(f1)] If $M\in \Mod R$ and $\Hom_R(\omega,M)=0$, then $M=0$.
\item[(f2)] If $N\in \Mod S^{op}$ and $\Hom_{S^{op}}(\omega,N)=0$, then $N=0$.
\end{enumerate}}
\end{df}

Typical examples of semidualizing bimodules include the free module of rank one, dualizing modules over a Cohen-Macaulay local ring
and the ordinary Matlis dual bimodule $_\Lambda D(\Lambda)_\Lambda$ of $_\Lambda \Lambda_\Lambda$ over an artin algebra $\Lambda$.
Any semidualizing bimodule over commutative rings is faithful (\cite[Proposition 3.1]{HW}).
Semidualizing bimodules occur in the literature with several different names,
e.g., in the work of \cite{Fox, Go, MR, Wa}.

Let $R$ be a commutative noetherian local ring with maximal ideal $\mathfrak{m}$ and residue
field $k=R/\mathfrak{m}$. According to \cite{Ku}, an artinian $R$-module $T$ is called \textbf{quasidualizing}
if the homothety $\hat{R}\rightarrow \Hom_R(T,T)$
is an isomorphism (where $\hat{R}$ is the $\mathfrak{m}$-adic completion of $R$) and $\Ext^{i\geqslant 1}_R(T,T)=0$.
It was proved in \cite[Lemma 3.11]{Ku} that if $L$ and $T$ are $R$-modules with $T$ quasidualizing such that $\Hom_R(T,L)=0$,
then $L=0$. Motivated by this result and \cite[Lemma 3.1]{HW}, an open question was posed in \cite{Ku} as follows.

\begin{Qu} \label{Qu: 2.2} {\rm (\cite[Question 3.12]{Ku})} Let $R$ be a commutative noetherian local ring.
If $L$ and $T$ are $R$-modules with $T$ quasidualizing such that $T\otimes_RL=0$, then $L=0$?
\end{Qu}

The following result shows that the answer to this question is negative in general.

\begin{prop} \label{prop: 2.3}  Let $R$ be a commutative noetherian complete local ring with maximal ideal $\mathfrak{m}$ and residue
field $k=R/\mathfrak{m}$. If $R$ is Gorenstein (that is, $\id_RR<\infty$) with $\dim R>0$, then
$E^0(R/\mathfrak{q})\otimes_RE^0(k)=0$ for any prime ideal $\mathfrak{q}$ with $\height(\mathfrak{q})>0$, where $\height(\mathfrak{q})$
is the height of $\mathfrak{q}$.
\end{prop}

\begin{proof} By \cite[Theorem 4.2]{M}, $E^0(k)$ is quasidualizing. Since $R$ is Gorenstein,
it follows from \cite[Fundamental Theorem]{Ba} that $E^i(R)=\oplus_{\height(\mathfrak{p})=i}E(R/\mathfrak{p})$
with $\mathfrak{p}\in\Spec(R)$ (the prime spectrum of $R$) for any $i\geqslant 0$. In particular,
$E^0(R)=\oplus_{\height(\mathfrak{p})=0}E(R/\mathfrak{p})$ with $\mathfrak{p}\in\Spec(R)$.
On the other hand, for any $\mathfrak{p},\mathfrak{q}\in \Spec(R)$ with $\height(\mathfrak{p})=0$
and $\height(\mathfrak{q})>0$, we have $\Hom_R(E^0(R/\mathfrak{q}), E^0(R/\mathfrak{p}))=0$. So
$\Hom_R(E^0(R/\mathfrak{q}), E^0(R))=0$ and $\Hom_R(E^0(R/\mathfrak{q}), R)=0$. Thus we have
\begin{align*}
&\ \ \ \ \ \Hom_R(E^0(R/\mathfrak{q})\otimes_RE^0(k), E^0(k))\\
& \cong\Hom_R(E^0(R/\mathfrak{q}),\Hom_R(E^0(k), E^0(k))) \ \text{(by the adjoint isomorphism theorem)}\\
& \cong\Hom_R(E^0(R/\mathfrak{q}),R)\ \text{(by \cite[Theorem 4.2]{M})}\\
& =0.
\end{align*}
Because $E^0(k)$ is an injective cogenerator for $\Mod R$, $E^0(R/\mathfrak{q})\otimes_RE^0(k)=0$.
\end{proof}

From now on, $_R\omega_S$ is a semidualizing bimodule.
For convenience, we write $(-)_*=\Hom_R(\omega,-)$ and $_R\omega^{\perp}=\{M\in\Mod R\mid\Ext_R^{i\geqslant 1}(\omega,M)=0\}$.

Let $M\in \Mod R$ and $N\in \Mod S$. Then we have the following two canonical valuation homomorphisms:
$$\theta_M:\omega\otimes_SM_*\rightarrow M$$ defined by $\theta_M(x\otimes f)=f(x)$ for any $x\in \omega$ and $f\in M_*$; and
$$\mu_N: N\rightarrow (\omega\otimes_SN)_*$$ defined by $\mu_N(y)(x)=x\otimes y$ for any $y\in N$ and $x\in \omega$.
Following \cite{W}, $M$ (resp. $N$) is called {\bf $\omega$-static} (resp. {\bf $\omega$-adstatic}) if $\theta_M$ (resp. $\mu_N$) is an isomorphism.
We denote by $\Stat(\omega)$ and $\Adst(\omega)$ the class of all $\omega$-static modules and the class of
all $\omega$-adstatic modules, respectively.

\begin{df} \label{df: 2.4} {\rm (\cite{HW}). The \textbf{Bass class} $\mathcal{B}_\omega(R)$
with respect to $\omega$ consists of all left $R$-modules $M$ satisfying
\begin{enumerate}
\item[(B1)] $M\in{_R\omega^{\perp}}$,
\item[(B2)] $\Tor_{\geqslant 1}^S(\omega,M_*)=0$, and
\item[(B3)] $M\in\Stat(\omega)$, that is, $\theta_M$ is an isomorphism in $\Mod R$.
\end{enumerate}
The {\bf Auslander class} $\mathcal
{A}_{\omega}(S)$ with respect to $\omega$ consists of all left $S$-modules $N$
satisfying
\begin{enumerate}
\item[(A1)] $\Tor^{S}_{i\geqslant 1}(\omega, N)=0$,
\item[(A2)] $\omega\otimes _{S}N\in{_R\omega^{\perp}}$, and
\item[(A3)] $N\in\Adst(\omega)$, that is, $\mu_{N}$ is an isomorphism in $\Mod S$.
\end{enumerate}}
\end{df}

\begin{df} \label{df: 2.5} {\rm (\cite{TH}). Let $M\in \Mod R$ and $n\geqslant 1$.
\begin{enumerate}
\item $\cTr_\omega M:=\Coker {f^0}_*$ is called the {\bf cotranspose} of $M$ with respect to $_R\omega_S$,
where $f^0$ is as in (1.1).
\item
$M$ is called
{\bf $n$-$\omega$-cotorsionfree} if $\Tor_{1\leqslant i\leqslant n}^S(\omega,\cTr_\omega M)=0$; and $M$ is called {\bf $\infty$-$\omega$-cotorsionfree}
if it is $n$-$\omega$-cotorsionfree for all $n$. The class of all $\infty$-$\omega$-cotorsionfree modules is denoted by ${\rm c}\mathcal{T}(R)$.
In particular, every module in $\Mod R$ is 0-$\omega$-cotorsionfree.
\end{enumerate}}
\end{df}
By \cite[Proposition 3.2]{TH}, a module is 2-$\omega$-cotorsionfree if and only if it is $\omega$-static.

Let $\mathcal{W}\subseteq\mathcal{X}$ be subclasses of $\Mod R$. Recall from \cite{AB} that
$\mathcal{W}$ is called a {\bf generator} for $\mathcal{X}$ if for any $X\in\mathcal{X}$, there exists an exact
sequence $0\rightarrow X^{'}\rightarrow W\rightarrow X\rightarrow 0$ in $\Mod R$ with $W\in\mathcal{W}$ and $X^{'}\in\mathcal{X}$;
$\mathcal{W}$ is called an {\bf $\Ext$-projective generator} for $\mathcal{X}$ if $\mathcal{W}$ is a generator for $\mathcal{X}$
and $\Ext_{R}^{i\geqslant 1}(W,X)=0$ for any $X\in \mathcal{X}$ and $W\in \mathcal{W}$. Also recall that $\mathcal{X}$ is called {\bf coresolving}
if it is closed under extensions, cokernels of monomorphisms and it contains all injective modules in $\Mod R$.

Let $M\in \Mod R$. An exact sequence (of finite or infinite length):
$$\cdots \to X_n \to \cdots \to X_1 \to X_0 \to M \to 0$$
in $\Mod R$ is called an {\bf $\mathcal{X}$-resolution} of
$M$ if all $X_i$ are in $\mathcal{X}$; furthermore, such an
$\mathcal{X}$-resolution is called {\bf proper}
if it remains exact after applying the functor $\Hom_R(X,-)$
for any $X\in \mathcal{X}$. The \textbf{$\mathcal{X}$-projective
dimension} $\mathcal{X}$-$\pd_RM$ of $M$ is defined as
$\inf\{n\mid$ there exists an $\mathcal
{X}$-resolution $0 \to X_n \to \cdots \to X_1\to X_0 \to M\to 0$
of $M$ in $\Mod R\}$. Dually, the notions of an \textbf{$\mathcal{X}$-coresolution},
an \textbf{$\mathcal{X}$-coproper coresolution}
and the \textbf{$\mathcal{X}$-injective dimension} $\mathcal{X}$-$\id_RM$ of $M$ are defined.

\begin{df} \label{df: 2.6}
{\rm (\cite{EJ}) A module $M\in \Mod R$ is
called \textbf{Gorenstein projective} if there exists an exact sequence of projective modules:
$${\bf P}:=\cdots \rightarrow P_1\rightarrow P_0\rightarrow P^{0}\rightarrow
P^{1}\rightarrow \cdots$$ in $\Mod R$ satisfying the conditions: (1) it remains exact after applying the functor
$\Hom_R(-,P)$ for any projective module $P$ in $\Mod R$; and (2) $M\cong\Im (P_0\rightarrow
P^{0})$. Dually, the notion of Gorenstein injective modules is defined. We use $\mathcal{GP}(R)$ (resp. $\mathcal{GI}(R)$)
to denote the subclass of $\Mod R$ consisting of Gorenstein projective (resp. Gorenstein injective) modules.}
\end{df}

\begin{Fa}\label{fa: 2.7}
{\rm \begin{enumerate}
\item[]
\item $\mathcal{B}_\omega(R)$ is
coresolving and $\Add _R\omega$ is an $\Ext$-projective generator for $\mathcal{B}_\omega(R)$ (see \cite[Proposition 5.1(b),Theorem 6.2]{HW} and \cite[Proposition 3.7]{TH}).
\item  When $\omega$ is a dualizing module over a local Cohen-Macaulay ring $R$, $\mathcal{B}_\omega(R)$ actually is exactly the class of modules admitting finite
Gorenstein injective dimensions (see \cite[Corollary 2.6]{En}). However, the following example illustrates that these two classes of modules are different in general.
\end{enumerate}}
\end{Fa}

\begin{exa} \label{exa: 2.8} {\rm Let $\Lambda$ be the finite-dimensional algebra over a field defined by the following quiver and relation:
\[\xymatrix{
 \circ 1 \ar@/^/[r]& \circ 2 \ar@/^/[l]& \circ 3\ar[l]&\circ 4\ar[l] \ar@/^/[r]&\circ 5, \ar@/^/[l]
  }\]
$$(\rad \Lambda)^2 = 0.$$
Take $_{\Lambda}\omega=_1^2$$\oplus$$^1$$_2$$^3\oplus$$3$$\oplus$$_3$$^4$$_5$$\oplus_4^5$ and $M= _3$$^4$$_5$.
Then by \cite[Example 3.1]{Wa1} and \cite[Theorem 3.9]{TH},
$_{\Lambda}\omega_{End(_{\Lambda}\omega)}$ is a semidualizing bimodule and $M\in\mathcal{B}_\omega(\Lambda)$.
But an easy computation shows that the Gorenstein injective dimension of $M$ is infinite.}
\end{exa}

Let $\mathcal{E}$ be a subcategory of an abelian category $\mathcal{A}$. Recall from
\cite{EJ} that a sequence:
$$\mathbb{S}: \cdots \to S_1 \to S_2 \to S_3 \to \cdots$$
in $\mathcal{A}$ is called {\bf
$\Hom_{\mathcal{A}}(\mathcal{E},-)$-exact} (resp. {\bf
$\Hom_{\mathcal{A}}(-,\mathcal{E})$-exact}) if {\bf
$\Hom_{\mathcal{A}}(E,\mathbb{S})$} (resp.
$\Hom_{\mathcal{A}}(\mathbb{S},E)$) is exact for any object $E$ in
$\mathcal{E}$. An epimorphism (resp. a monomorphism) in
$\mathcal{A}$ is called {\bf $\mathcal{E}$-proper} (resp. {\bf
$\mathcal{E}$-coproper}) if it is
$\Hom_{\mathcal{A}}(\mathcal{E},-)$-exact (resp.
$\Hom_{\mathcal{A}}(-,\mathcal{E})$-exact).

\begin{df} \label{df: 2.9} {\rm (\cite{H2}) Let $\mathcal{E}$ and $\mathcal{T}$ be
subcategories of an abelian category $\mathcal{A}$. Then $\mathcal{T}$ is called {\bf
$\mathcal{E}$-coresolving} in $\mathcal{A}$ if the following
conditions are satisfied.
\begin{enumerate}
\item $\mathcal{T}$ admits an {\bf $\mathcal{E}$-coproper cogenerator} $\mathcal{C}$, that is,
$\mathcal{C}\subseteq \mathcal{T}$, and for any object $T$ in $\mathcal{T}$, there exists a
$\Hom_{\mathcal{A}}(-,\mathcal{E})$-exact exact sequence
$0\to T\to C \to T^{'} \to 0$ in
$\mathcal{A}$ such that $C$ is an object in $\mathcal{C}$ and
$T^{'}$ is an object in $\mathcal{T}$.
\item $\mathcal{T}$ is {\bf closed under $\mathcal{E}$-coproper
extensions}, that is, for any
$\Hom_{\mathcal{A}}(-,\mathcal{E})$-exact exact sequence $0\to
A_1\to A_2 \to A_3 \to 0$ in $\mathcal{A}$, if both $A_1$ and $A_3$
are objects in $\mathcal{T}$, then $A_2$ is also an object in
$\mathcal{T}$.
\item $\mathcal{T}$ is {\bf closed under cokernels of
$\mathcal{E}$-coproper monomorphisms}, that is, for any
$\Hom_{\mathcal{A}}(-,\mathcal{E})$-exact exact sequence $0\to
A_1\to A_2 \to A_3 \to 0$ in $\mathcal{A}$, if both $A_1$ and $A_2$
are objects in $\mathcal{T}$, then $A_3$ is also an object in
$\mathcal{T}$.
\end{enumerate}

Dually, the notions of {\bf $\mathcal{E}$-proper generators} and {\bf $\mathcal{E}$-resolving
subcategories} are defined.}
\end{df}

\section {\bf Relative Homological Dimensions}

Holm and White obtained in \cite{HW} some equivalent characterizations of $\mathcal{B}_\omega(R)$ in terms of
the so-called $\omega$-projective and $\omega$-flat modules. Similar results were also proved by Enochs
and Holm in \cite{EH}. Recently, we proved in
\cite[Theorem 3.9]{TH} that $\mathcal{B}_\omega(R)={\rm c}\mathcal{T}(R)\cap {_R\omega^{\bot}}$. In the beginning of this section,
we investigate the further relation among ${\rm c}\mathcal{T}(R),{_R\omega^{\bot}}$ and $\mathcal{B}_\omega(R)$.

\begin{prop} \label{prop: 3.1}
\begin{enumerate}
\item[]
\item If $\pd_R\omega<\infty$, then ${\rm c}\mathcal{T}(R)\subseteq{_R\omega^{\perp}}$.
\item If $\pd_{S^{op}}\omega<\infty$, then ${_R\omega^{\perp}}\subseteq{\rm c}\mathcal{T}(R)$.
\end{enumerate}
\end{prop}

\begin{proof} (1) Let $M\in{\rm c}\mathcal{T}(R)$. Then by \cite[Proposition 3.7]{TH}, there exists an exact sequence:
$$\cdots \rightarrow W_{n}\rightarrow W_{n-1}\rightarrow \cdots \rightarrow W_0\rightarrow M\rightarrow 0$$
in $\Mod R$ with all $W_{i}\in\Add_R\omega$. Put $M_i=\Im(W_i\rightarrow W_{i-1})$ for any $i\geqslant 1$. We may assume $\pd_R\omega=n<\infty$
by assumption. Since $W_i\in {_R\omega^{\perp}}$ by \cite[Lemma 2.5(1)]{TH}, $\Ext_R^i(\omega,M)\cong\Ext_R^{i+n}(\omega,M_n)=0$
for any $i\geqslant 1$ and $M\in{_R\omega^{\perp}}$.

(2) Let $M\in {_R\omega^{\perp}}$ and $\pd_{S^{op}}\omega=n<\infty$. Then we get an exact sequence:
$$0\rightarrow \coOmega^{i}(M)_*\rightarrow  I^i(M)_*\rightarrow \coOmega^{i+1}(M)_*\rightarrow 0$$
in $\Mod S$ for any $i\geqslant 0$. Note that $\fd_{S^{op}}\omega=\pd_{S^{op}}\omega=n$ because $\omega$ is finitely presented as a right $S$-module.
Since $\Tor^S_{i\geqslant 1}(\omega,I_*)=0$ for any injective left $R$-module $I$ by \cite[Lemma 2.5(2)]{TH},
$\Tor^S_{j}(\omega,\coOmega^{i}(M)_*)
%\cong \Tor^S_{j+1}(\omega,\coOmega^{i+1}(M)_*)\cong\cdots
\cong\Tor^S_{j+n}(\omega,\coOmega^{i+n}(M)_*)=0$
for any $i\geqslant 0$ and $j\geqslant 1$; in particular, $\Tor^S_{1}(\omega,\coOmega^{2}(M)_*)=0$. Then we have the following diagram with exact rows:
$$\xymatrix{0 \ar[r]
& \omega\otimes_S\coOmega^{1}(M)_* \ar[r] \ar[d]^{\theta_{\coOmega^{1}(M)}}
& \omega\otimes_S{I^1(M)}_{*} \ar[d]^{\theta_{I^1(M)}}\\
0 \ar[r]&  \coOmega^{1}(M) \ar[r] & I^1(M). }$$
Because $\theta_{I^1(M)}$ is an isomorphism by \cite[Lemma 2.5(2)]{TH}, $\theta_{\coOmega^{1}(M)}$ is a monomorphism.
So $\coOmega^{1}(M)$ is 2-$\omega$-cotorsionfree by \cite[Lemma 4.1(1)]{TH}.
On the other hand, because $\Tor^S_{1}(\omega,\coOmega^{1}(M)_*)$ = 0 by the above argument, we have the following commutative diagram
with exact rows:
$$\xymatrix{0 \ar[r]&  \omega\otimes_SM_* \ar[r] \ar[d]^{\theta_{M}}
& \omega\otimes_S{I^0(M)}_{*} \ar[r] \ar[d]^{\theta_{I^0(M)}}& \omega\otimes_S\coOmega^{1}(M)_* \ar[d]^{\theta_{\coOmega^{1}(M)}}\ar[r]& 0\\
0 \ar[r]&  M \ar[r] & I^0(M) \ar[r] &\coOmega^{1}(M)\ar[r] & 0. }$$
Because $\theta_{I^0(M)}$ is an isomorphism by \cite[Lemma 2.5(2)]{TH}, applying the snake lemma we have that $\theta_M$
is also an isomorphism and $M$ is 2-$\omega$-cotorsionfree. So by \cite[Corollary 3.8]{TH}, there exists an exact sequence
$0\rightarrow M_1\rightarrow W_0\rightarrow M\rightarrow 0$ in $\Mod R$ with $W_0\in \Add_R\omega$ and $\Ext_R^1(\omega,M_1)=0$.
Thus $M_1\in {_R\omega^{\perp}}$ since $M\in {_R\omega^{\perp}}$. Then by a argument similar to the above,
we get an exact sequence
$0\rightarrow M_2\rightarrow W_1\rightarrow M_1\rightarrow 0$ in $\Mod R$ with $W_1\in \Add_R\omega$ and $M_2\in {_R\omega^{\perp}}$.
Continuing this procedure, we get a proper $\Add_R\omega$-resolution:
$$\cdots \rightarrow W_{n}\rightarrow W_{n-1}\rightarrow \cdots \rightarrow W_0\rightarrow M\rightarrow 0$$
of $M$ in $\Mod R$. Thus $M\in{\rm c}\mathcal{T}(R)$ by \cite[Proposition 3.7]{TH}.
\end{proof}

The following result extends \cite[Corollary 2.16]{Wa1}.

\begin{cor} \label{cor: 3.2}
\begin{enumerate}
\item[]
\item If $\pd_R\omega<\infty$, then $\mathcal{B}_\omega(R)={\rm c}\mathcal{T}(R)$.
\item If $\pd_{S^{op}}\omega<\infty$, then $\mathcal{B}_\omega(R)={_R\omega^{\perp}}$.
\end{enumerate}
\end{cor}

\begin{proof}
It is an immediate consequence of Proposition 3.1 and \cite[Theorem 3.9]{TH}.
\end{proof}

We write $\Ker\Ext^{i\geqslant 1}_S(-,\omega^+)=\{N\in\Mod S\mid\Ext^{i\geqslant 1}_S(N, \omega^+)=0$\} and
$\mathcal{H}(\omega)=\Adst(\omega)\bigcap\Ker\Ext^{i\geqslant 1}_S(-,\omega^+)$,
where $(-)^+=\Hom_\mathbb{Z}(-,\mathbb{Q}/\mathbb{Z})$
with $\mathbb{Z}$ the additive group of integers and $\mathbb{Q}$ the additive group of rational numbers.
In the following result, we provide a viewpoint from Morita equivalence for ${\rm c}\mathcal{T}(R)$.

\begin{thm} \label{thm: 3.3} There exists an equivalence of
categories: \\
\[\begin{array}{cccc}
\xymatrix{ {\rm c}\mathcal{T}(R) \ar@<1ex>[rr]^{(-)_*} &&
\mathcal{H}(\omega).\ar@<1ex>[ll]^{{\omega\otimes_S-}}_{\thicksim}}
\end{array}
\]
\end{thm}

\begin{proof} According to \cite[2.4]{W}, the functors $(-)_*$ and $\omega\otimes_S-$ induce an equivalence between the category of all
2-$\omega$-cotorsionfree modules and $\Adst(\omega)$. So it suffices to show that $(-)_*$ (resp. ${\omega\otimes_S-}$)
maps ${\rm c}\mathcal{T}(R)$ (resp. $\mathcal{H}(\omega)$) to $\mathcal{H}(\omega)$ (resp. ${\rm c}\mathcal{T}(R)$).

Let $M\in {\rm c}\mathcal{T}(R)$. Then by \cite[2.4]{W}, we have $M_*\in \Adst(\omega)$. By \cite[Proposition 3.7]{TH}
there exists a proper $\Add_R\omega$-resolution:
$$\cdots \rightarrow W_{n}\rightarrow W_{n-1}\rightarrow \cdots \rightarrow W_0\rightarrow M\rightarrow 0  \eqno(3.1)$$
of $M$ in $\Mod R$. Thus we get an exact sequence:
$$\cdots \rightarrow {W_{n}}_{*}\rightarrow {W_{n-1}}_{*}\rightarrow \cdots \rightarrow {W_{0}}_{*}\rightarrow M_{*}\rightarrow 0$$
in $\Mod S$. Applying $\omega\otimes_S-$ to this exact sequence gives back the sequence (3.1). Then we get easily that
$\Tor^S_{i\geqslant 1}(\omega,M_*)$ = 0 because $\Tor^S_{i\geqslant 1}(\omega,{W_j}_*)=0$ for any $j\geqslant 0$ by \cite[Lemma 2.5(2)]{TH}.
It follows from the mixed isomorphism theorem that $\Ext_S^{i\geqslant 1}(M_*,\omega^+)\cong[\Tor_{i\geqslant 1}^S(\omega,M_*)]^+=0$.
So $M_*\in\Ker\Ext^{i\geqslant 1}_S(-,\omega^+)$ and $M_*\in \mathcal{H}(\omega)$.

Conversely, let $N\in \mathcal{H}(\omega)$. Then $(\omega\otimes_SN)_*\cong N$. It follows from the mixed isomorphism theorem that
$[\Tor_{i\geqslant 1}^S(\omega,(\omega\otimes_SN)_*)]^+\cong[\Tor_{i\geqslant 1}^S(\omega,N)]^+
\cong\Ext_S^{i\geqslant 1}(N,\omega^+)$ $=0$ and $\Tor_{i\geqslant 1}^S(\omega,(\omega\otimes_SN)_*)=0$. In addition,
$\omega\otimes_SN$ is 2-$\omega$-cotorsionfree by \cite[2.4]{W}.
%Applying $\Hom_R(\omega,-)$ to the minimal injective resolution of $\omega\otimes_SM$ yields an exact sequence:
%\footnotesize{$$0\rightarrow (\omega\otimes_SM)_*(\cong M)\rightarrow {I^0(\omega\otimes_SM)}_*\rightarrow {I^1(\omega\otimes_SM)}_*\rightarrow \cTr_\omega(\omega\otimes_SM)\rightarrow 0.$$}
%So $\Tor_{i}^S(\omega,\cTr_\omega(\omega\otimes_SM))\cong\Tor_{i-2}^S(\omega,M)$ for any $i\geqslant 3$ by \cite[Lemma 2.5(2)]{TH}.
Thus we conclude that $\omega\otimes_SN$ is $\infty$-$\omega$-cotorsionfree by \cite[Corollary 3.4]{TH}.
\end{proof}

Following \cite{HW}, set
$$\mathcal{F}_\omega(R)=\{\omega\otimes_SF\mid F\ {\rm \ is\ flat\ in}\ \Mod S\},$$
$$\mathcal{P}_\omega(R)=\{\omega\otimes_SP\mid P\ {\rm \ is\ projective\ in}\ \Mod S\},$$
$$\mathcal{I}_\omega(S)=\{\Hom_R(\omega,I)\mid I\ {\rm \ is\ injective\ in}\ \Mod R\}.$$
The modules in $\mathcal{F}_\omega(R)$, $\mathcal{P}_\omega(R)$ and $\mathcal{I}_\omega(S)$ are called \textbf{$\omega$-flat},
\textbf{$\omega$-projective} and \textbf{$\omega$-injective} respectively. For a module $M\in\Mod R$, we
use $\limit_M(R)$ to denote the subcategory of $\Mod R$ consisting of all modules isomorphic to direct summands of a
direct limit of a family modules in which each is a finite direct
sum of copies of $M$.

\begin{prop} \label{prop: 3.4}
\begin{enumerate}
\item[]
\item$\mathcal{F}_{\omega}(R)=\limit_{\omega}(R)$
\item$\mathcal{P}_{\omega}(R)=\Add_{R}\omega$.
\item$\mathcal{I}_{\omega}(S)=\Prod_SE_*$ with $_RE$ an injective cogenerator for $\Mod R$.
\end{enumerate}
\end{prop}

\begin{proof} (1) It is well known that a module in $\Mod S$ is flat if and only if it is in $\limit_S(S)$.
Because the functor $\omega\otimes_S-$ commutes with direct limits, we get easily
$\mathcal{F}_{\omega}(R)\subseteq\limit_{\omega}(R)$. Now let $M\in\limit_{\omega}(R)$. Then $M\in \mathcal{B}_\omega(R)$
by \cite[Proposition 4.2(a)]{HW}. Because $_R\omega$ admits a degreewise finite
$R$-projective resolution, $\Hom_R(\omega,-)$ commutes with direct limits. So $\Hom_R(\omega,M)$ is in $\limit_S(S)$,
that is, $\Hom_R(\omega,M)$ is a flat left $S$-module. Then by \cite[Lemma 5.1(a)]{HW}, we have $M\in \mathcal{F}_\omega(R)$,
and thus $\limit_{\omega}(R)\subseteq\mathcal{F}_{\omega}(R)$.

(2) and (3) See \cite[Proposition 2.4]{LHX}.
\end{proof}

The following result establishes the relation between the relative homological dimensions of a module $M$
and the corresponding standard homological dimensions of $M_*$. It extends \cite[Theorem 2.11]{TW}.

\begin{thm} \label{thm: 3.5}
\begin{enumerate}
\item[]
\item $\fd_SM_*\leqslant\mathcal{F}_\omega(R)$-$\pd_RM$ for any $M\in\Mod R$, the equality holds if $M\in{\rm c}\mathcal{T}(R)$.
\item$\pd_SM_*\leqslant\mathcal{P}_\omega(R)$-$\pd_RM$ for any $M\in\Mod R$, the equality holds if $M\in{\rm c}\mathcal{T}(R)$.
\item$\id_R\omega\otimes_SN\leqslant\mathcal{I}_\omega(S)$-$\id_SN$ for any $N\in\Mod S$, the equality holds if $N\in\mathcal{A}_\omega(S)$.
\end{enumerate}
\end{thm}

\begin{proof}
(1) Let $M\in \Mod R$ with $\mathcal{F}_\omega(R)$-$\pd_RM=n <\infty$. Then there
exists an exact sequence:
$$0 \to L_{n} \to \cdots \to L_{1} \to L_{0} \to M \to 0 \eqno{(3.2)}$$
in $\Mod R$ with all $L_{i}$ in $\limit_\omega(R)$ by Proposition 3.4(1). Because $_R\omega$ admits a degreewise finite
$R$-projective resolution, $\Ext_R^i(\omega,-)$ commutes with direct limits for any $i\geqslant 0$.
Also notice that $(_R\omega)_*\cong S$ and $\omega\in{_R\omega^{\bot}}$, so we have that ${L_{i}}_*$
is in $\limit_S(S)$ (that is, ${L_{i}}_*$ is left $S$-flat) and $L_i\in{_R\omega^{\bot}}$
for any $0 \leqslant i \leqslant n$. Applying the functor $\Hom_{R}(\omega,-)$
to the exact sequence (3.2) we obtain the following exact sequence:
$$0 \to {L_n}_* \to \cdots \to {L_1}_* \to {L_0}_* \to M_* \to 0$$
in $\Mod S$, and so $\fd_S M_*\leqslant n$.

(2) Let $M\in \Mod R$ with $\mathcal{P}_\omega(R)$-$\pd_RM=n<\infty$. Then there exists an exact sequence:
$$0 \to \omega_{n} \to \cdots \to \omega_{1} \to \omega_{0} \to M \to 0 \eqno{(3.3)}$$
in $\Mod R$ with all $\omega_{i}\in\Add_R\omega$ by Proposition 3.4(2). Because
all ${\omega_i}_*$ are projective left $S$-modules and $\Add_R\omega\subseteq{_R\omega^{\perp}}$ by \cite[Lemma 2.5(1)]{TH},
applying the functor $(-)_*$ to the exact
sequence (3.3) we get the following exact sequence:
$$0 \to {\omega_n}_* \to \cdots \to {\omega_{1}}_* \to {\omega_{0}}_* \to {M_*} \to 0$$ in $\Mod S$,
and so $\pd_SM_*\leqslant n$.

Now suppose $M\in{\rm c}\mathcal{T}(R)$. Then $\omega\otimes _{S} M_*\cong M$.
By \cite[Corollary 3.4(3)]{TH}, we have $\Tor_{i\geqslant 1}^{S}(\omega, M_*)=0$.
We will prove the equalities in (1) and (2) hold.

(1) Assume $\fd_SM_*=n<\infty$. Then there exists an exact sequence:
$$0 \to F_{n} \to \cdots \to F_{1} \to F_{0} \to M_*\to 0$$
in $\Mod S$ with all $F_{i}$ flat. Applying the functor $\omega\otimes _{S}-$
to it we get an exact sequence:
$$0 \to \omega \otimes _{S}F_{n} \to \cdots \to \omega \otimes _{S}F_{1} \to \omega \otimes _{S}F_{0}
\to \omega\otimes _{S} M_* (\cong M)\to 0$$
in $\Mod R$ with all $\omega \otimes _{S}F_{i}$ in $\mathcal{F}_\omega(R)$, so we have $\mathcal{F}_\omega(R)$-$\pd_RM\leqslant n$.

(2) Assume $\pd_SM_*=n<\infty$. Then there exists an exact sequence:
$$0 \to P_n \to \cdots \to P_1 \to P_0\to M_* \to 0$$
in $\Mod S$ with all $P_i$ projective. Applying the functor $\omega\otimes _{S}-$
to it we get an exact sequence:
$$0 \to \omega \otimes _{S}P_{n} \to \cdots \to \omega \otimes _{S}P_{1} \to \omega \otimes _{S}P_{0}
\to \omega\otimes _{S} M_*(\cong M) \to 0$$
in $\Mod R$ with all $\omega \otimes _{S}P_{i}$ in $\mathcal{P}_\omega(R)$, and so $\mathcal{P}_\omega(R)$-$\pd_RM\leqslant n$.

(3) Let $N\in\Mod S$ with $\mathcal{I}_\omega(S)$-$\id_SN=n<\infty$ and $_RE$ be an injective cogenerator
for $\Mod R$. Then there exists an exact sequence:
$$0 \to N \to I^0 \to I^1 \to\cdots \to I^n \to 0 \eqno{(3.4)}$$
in $\Mod S$ with all $I^{i}$ in $\Prod_SE_*$ by Proposition 3.4(3). Because $\omega_S$ admits a degreewise finite
$S$-projective resolution, $\Tor_j^S(\omega,-)$ commutes with direct products for any $j\geqslant 0$. Then by \cite[Lemma 2.5(2)]{TH},
$\omega\otimes_SI^i(\in\Prod_RE)$ is injective in $\Mod R$ and $\Tor_{j\geqslant 1}^S(\omega,I^i)=0$ for any $0\leqslant i\leqslant n$.
Applying the functor $\omega\otimes_S-$ to the exact sequence (3.4) we obtain the following exact sequence:
$$0 \to \omega\otimes_SN \to \omega\otimes_SI^0 \to \omega\otimes_SI^1 \to\cdots \to \omega\otimes_SI^n \to 0$$
in $\Mod R$, and so $\id_R\omega\otimes_SN\leqslant n$.

Now suppose $N\in\mathcal{A}_\omega(S)$. Then $N\cong (\omega\otimes _{S}N)_*$ and $\omega\otimes _{S}N\in{_R\omega^{\bot}}$.
If $\id_R\omega\otimes_SN=n<\infty$, then there exists an exact sequence:
$$0 \to \omega\otimes_SN \to E^0 \to E^1 \to\cdots \to E^n \to 0$$
in $\Mod R$ with all $E^i$ injective. Applying the functor $\Hom_R(\omega,-)$
to it we get an exact sequence:
$$0 \to (\omega\otimes_SN)_*(\cong N) \to {E^0}_* \to {E^1}_* \to\cdots \to {E^n}_* \to 0$$
in $\Mod S$ with all ${E^i}_*\in \mathcal{I}_\omega(S)$, and so $\mathcal{I}_\omega(S)$-$\id_SN\leqslant n$.
\end{proof}

For a subclass $\mathcal{X}$ of $\Mod R$, we write $\id_R\mathcal{X}:=\sup\{\id_RX\mid X\in \mathcal{X}\}$. As an application
of Theorem 3.5, we get the following

\begin{prop} \label{prop: 3.6}
\begin{enumerate}
\item[]
\item $\sup\{\mathcal{F}_\omega(R)$-$\pd_RM\mid M\in {\rm c}\mathcal{T}(R)\ with \ \mathcal{F}_\omega(R)$-$\pd_RM<\infty\}\leqslant\id_R\mathcal{F}_\omega(R)$.
\item $\sup\{\mathcal{P}_\omega(R)$-$\pd_RM\mid M\in {\rm c}\mathcal{T}(R)\ with \ \mathcal{P}_\omega(R)$-$\pd_RM<\infty\}\leqslant\id_R\mathcal{P}_\omega(R)$.
\end{enumerate}
\end{prop}

\begin{proof} (1) Let $\id_R\mathcal{F}_\omega(R)=n<\infty$ and $M\in {\rm c}\mathcal{T}(R)$ with $\mathcal{F}_\omega(R)$-$\pd_RM=m<\infty$.
By Theorem 3.5(1), $\fd_SM_*=m$ and there exists an exact sequence:
$$0 \to F_m \to Q_{m-1} \to \cdots \to Q_1 \to Q_0
\to M_* \to 0 \eqno{(3.5)}$$ in $\Mod S$ with $F_m$ flat and all $Q_i$
projective. Because $\omega\otimes_SM_*\cong M$ and $\Tor_{j\geqslant 1}^{S}(\omega, M_*)$ $=0$ by \cite[Corollary 3.4(3)]{TH},
applying the functor $\omega\otimes _S-$ to the exact sequence (3.5), we
get the following exact sequence:
$$0 \to \omega\otimes_SF_m \to \omega\otimes_SQ_{m-1} \to \cdots \to \omega\otimes_SQ_1 \to \omega\otimes_SQ_0
\to \omega\otimes_SM_*(\cong M) \to 0 \eqno{(3.6)}$$
in $\Mod R$ with $\omega\otimes_SF_m$ in $\mathcal{F}_\omega(R)$ ($=\limit_\omega(R)$ by Proposition 3.4(1))
and all $\omega\otimes_SQ_{i}$ in $\mathcal{P}_\omega(R)$ ($=\Add_R\omega$ by Proposition 3.4(2)). Notice that
$_R\omega$ admits a degreewise finite $R$-projective resolution and $\omega\in {_R\omega^{\bot}}$,
so $\Ext^{j\geqslant 1}_R(\omega\otimes_SQ_{i}, \omega\otimes_SF_m)=0$ for any $0\leqslant i \leqslant m-1$.

Suppose $m>n$. Because $\id_R\omega\otimes_SF_m\leqslant n$, it follows from the exact
sequence (3.6) that $\Ext_R^1(K, \omega\otimes_SF_m)\cong \Ext_R^m(M, \omega\otimes_SF_m)=0$, where
$K=\Coker (\omega\otimes_SF_m \to \omega\otimes_SQ_{m-1})$. Thus the exact sequence
$0 \to \omega\otimes_SF_m \to \omega\otimes_SQ_{m-1} \to K \to 0$ splits and
$K\in \mathcal{P}_\omega(R)(\subseteq \mathcal{F}_\omega(R))$. It induces that $\mathcal{F}_\omega(R)$-$\pd_RM\leqslant m-1$,
which is a contradiction. Thus we conclude that $m\leqslant n$.

(2) It is similar to the proof of (1), so we omit it.
\end{proof}

Note that $_RR_R$ is a semidualizing bimodule. Let $R$ be a left noetherian ring and $_R\omega_S={_RR_R}$. Then we have
the following facts:
\begin{enumerate}
\item $\mathcal{F}_\omega(R)$ and $\mathcal{P}_\omega(R)$
are the subclasses of $\Mod R$ consisting of flat modules and projective modules respectively,
and $\mathcal{F}_\omega(R)$-$\pd_RM=\fd_RM$ and $\mathcal{P}_\omega(R)$-$\pd_RM=\pd_RM$ for any $M\in \Mod R$.
\item $\id_R\mathcal{F}_\omega(R)=\id_RR$ and $\id_R\mathcal{P}_\omega(R)=\id_RR$ by \cite[Theorem 1.1]{Ba1}.
\item ${\rm c}\mathcal{T}(R)=\Mod R$ by \cite[Proposition 3.7]{TH}.
\end{enumerate}
So by Proposition 3.6, we immediately have the following

\begin{cor} \label{cor: 3.7} For a left noetherian ring $R$, we have
\begin{enumerate}
\item $\sup\{\fd_RM\mid M\in \Mod R$ with $\fd_RM<\infty\}\leqslant \id_RR$.
\item {\rm (\cite[Proposition 4.3]{Ba1})} $\sup\{\pd_RM\mid M\in \Mod R$ with $\pd_RM<\infty\}\leqslant \id_RR$.
\end{enumerate}
\end{cor}

In the rest of this section, for a module $M\in \Mod R$, in case $\mathcal{P}_\omega(R)$-$\pd_RM<\infty$, we establish the relation
between $\mathcal{P}_\omega(R)$-$\pd_RM$
and some standard homological dimensions of related modules.

\begin{lem} \label{lem: 3.8}
If $M\in{\rm c}\mathcal{T}(R)$ and $N\in{_R\omega^{\perp}}$, then for any $i\geqslant 0$, we have an isomorphism of abelian groups:
$$\Ext_R^i(M,N)\cong\Ext_S^i(M_*,N_*).$$
\end{lem}

\begin{proof}
We proceed by induction on $i$.

Let $i=0$. Since $M\in{\rm c}\mathcal{T}(R)$, $\omega\otimes_SM_*\cong M$. It follows from the adjoint isomorphism theorem that $\Hom_R(M,N)\cong
\Hom_R(\omega\otimes_SM_*,N)\cong$ $\Hom_S(M_*,N_*)$. Indeed the isomorphism is natural in $M$ and $N$.

Now suppose $i\geqslant 1$. The induction hypothesis implies that there exists a natural isomorphism:
$$\Ext^j_R(L,H)\cong\Ext^j_S(L_*,H_*)$$
for any $L\in{\rm c}\mathcal{T}(R)$, $H\in{_R\omega^{\perp}}$ and $0\leqslant j\leqslant i-1$. Because $N\in{_R\omega^{\perp}}$ by assumption,
$\coOmega^1(N)\in {_R\omega^{\perp}}$ and we have an exact sequence:
$$0\rightarrow N_*\rightarrow {I^0(N)}_*\rightarrow {\coOmega^1(N)}_*\rightarrow 0.$$
Applying the functor $\Hom_S(M_*,-)$ to it yields a commutative diagram with exact rows:
{\footnotesize$$\xymatrix{\Ext_R^{i-1}(M,I^0(N)) \ar[r] \ar[d] & \Ext_R^{i-1}(M,\coOmega^1(N)) \ar[r] \ar[d]
& \Ext_R^{i}(M,N) \ar[r] \ar[d]& 0\\
\Ext_S^{i-1}(M_*,{I^0(N)}_*) \ar[r]& \Ext_S^{i-1}(M_*,{\coOmega^1(N)}_*) \ar[r] & \Ext_S^{i}(M_*,N_*) \ar[r] &\Ext_S^{i}(M_*,{I^0(N)}_*). }$$}
By the induction hypothesis, the first two columns in the above diagram are natural isomorphisms.
Since $M\in{\rm c}\mathcal{T}(R)$ by assumption, by the mixed isomorphism theorem and \cite[Corollary 3.4(3)]{TH}
we have $\Ext_S^{i}(M_*,{I^0(N)}_*)\cong\Hom_R(\Tor_S^{i}(\omega,M_*),I^0(N))=0$.
It follows that $\Ext_R^i(M,N)\cong\Ext_S^i(M_*,N_*)$ naturally.
\end{proof}

We also need the following criterion.

\begin{lem} \label{lem: 3.9} Let $M\in \Mod R$ admit a degreewise finite $R$-projective resolution. If $\mathcal{P}_\omega(R)$-$\pd_RM<\infty$,
then $\mathcal{P}_\omega(R)$-$\pd_RM=\sup\{i\geqslant 0\mid\Ext_R^i(M,\omega)\neq 0\}$.
\end{lem}

\begin{proof}
Let $\mathcal{P}_\omega(R)$-$\pd_RM=n<\infty$ and
$$0\rightarrow \omega_n\rightarrow\cdots \rightarrow \omega_{1}\rightarrow \omega_0\rightarrow M\rightarrow 0$$
be an exact sequence in $\Mod R$ with all $\omega_i$ in $\mathcal{P}_\omega(R)(=\Add_R\omega)$. It is easy to see that $\Ext_R^{i}(M,\omega)=0$ for $i\geqslant n+1$.
Put $M_{n-1}=\Coker(\omega_{n}\rightarrow \omega_{n-1})$.

If $\Ext_R^n(M,\omega)=0$, then by \cite[Lemma 3.1.6]{GT}, we have that $\Ext_R^n(M,\omega_i)=0$ and $\Ext_R^{\geqslant 1}(\omega_j,\omega_i)=0$
for any $0\leqslant i,j \leqslant n$. So $\Ext_R^1(M_{n-1},\omega_n)\cong \Ext_R^n(M,\omega_n)=0$ and the exact sequence:
$$0\rightarrow \omega_n\rightarrow \omega_{n-1}\rightarrow M_{n-1}\rightarrow 0$$ splits. It implies that $M_{n-1}\in \mathcal{P}_\omega(R)$
and $\mathcal{P}_\omega(R)$-$\pd_RM\leqslant n-1$, which is a contradiction. So we conclude that $\Ext_R^n(M,\omega)\neq 0$.
\end{proof}

Now we are in a position to give the following

\begin{prop} \label{prop: 3.10} Let $M\in \Mod R$ admit a degreewise finite $R$-projective resolution.
If $\mathcal{P}_\omega(R)$-$\pd_RM<\infty$, then
$\mathcal{P}_\omega(R)$-$\pd_RM\leqslant\min\{\id_R\omega, \id_SS, \pd_RM, \pd_SM_*\}$.
\end{prop}

\begin{proof} Let $M\in \Mod R$ with $\mathcal{P}_\omega(R)$-$\pd_RM<\infty$. Then $M\in{\rm c}\mathcal{T}(R)$
by \cite[Proposition 3.7]{TH}. So $\Ext_R^i(M,\omega)\cong\Ext_S^i(M_*, \omega_*)\cong\Ext_S^i(M_*, S)$
for any $i\geqslant 0$ by Lemma 3.8, and hence $\sup\{i\geqslant 0\mid\Ext_R^i(M,\omega)\neq 0\} \leqslant\min
\{ \id_R\omega, \id_SS, \pd_RM, \pd_SM_*\}$. Now the assertion follows from Lemma 3.9.
\end{proof}

The following example shows that the finiteness of $\mathcal{P}_\omega(R)$-$\pd_RM$
is necessary for the conclusion of Proposition 3.10.

\begin{exa} \label{exa: 3.11} {\rm Let $G$ be a finite group and $k$ a field such that the characteristic
of $k$ divides $|G|$. Take $R=S=\omega=kG$. By \cite[Theorem 3.3 and Propositon 3.10]{ARS},
the group algebra $kG$ is a non-semisimple symmetric artin algebra. Then $\id_{R}\omega=0$
and there exists a $kG$-module $M$ with $\mathcal{P}_\omega(R)$-$\pd_RM$ infinite.}
\end{exa}

\section {\bf The Bass Injective Dimension of Modules}

For a module $M$ in $\Mod R$, we study in this section the properties of the {\bf Bass injective dimension} $\mathcal{B}_\omega(R)$-$\id_RM$ of $M$.
We begin with the following easy observation.

\begin{lem} \label{lem: 4.1} For any $M\in \Mod R$, if $\mathcal{B}_\omega(R)$-$\id_RM<\infty$ and $M\in {_R\omega^{\bot}}$,
then $M\in\mathcal{B}_\omega(R)$.
\end{lem}

\begin{proof}
It is easy to get the assertion by using induction on $\mathcal{B}_\omega(R)$-$\id_RM$.
\end{proof}

Now we give some criteria for computing $\mathcal{B}_\omega(R)$-$\id_RM$ in terms of the vanishing of
Ext-functors and some special approximations of $M$.

\begin{thm} \label{thm: 4.2} Let $M\in\Mod R$ with $\mathcal{B}_\omega(R)$-$\id_RM<\infty$ and $n\geqslant 0$. Then the following statements
are equivalent.
\begin{enumerate}
\item $\mathcal{B}_\omega(R)$-$\id_RM\leqslant n$.
\item $\coOmega^m(M)\in \mathcal{B}_\omega(R)$ for $m\geqslant n$.
\item $\Ext_{R}^{\geqslant n+1}(\omega,M)=0$.
\item There exists an exact sequence: $$0\rightarrow M\rightarrow X^M\rightarrow W^M\rightarrow 0$$ in $\Mod R$ such that $X^M\in \mathcal{B}_\omega(R)$ and $\mathcal{P}_\omega(R)$-$\id_RW^M\leqslant n-1$.
\item There exists an exact sequence: $$0\rightarrow X_M\rightarrow W_M\rightarrow M\rightarrow 0$$ in $\Mod R$ such that $X_M\in \mathcal{B}_\omega(R)$ and $\mathcal{P}_\omega(R)$-$\id_RW_M\leqslant n$.
\end{enumerate}
\end{thm}

\begin{proof} $(1)\Rightarrow (2)$ follows from \cite[Theorem 6.2]{HW} and \cite[Theorem 4.8]{H2}, $(2)\Rightarrow (3)$ follows from the dimension shifting, and
$(4)\Rightarrow (1)$ follows from the fact that $\mathcal{P}_\omega(R)\subseteq\mathcal{B}_\omega(R)$.

$(3)\Rightarrow (1)$ Let $M\in\Mod R$ with $\mathcal{B}_\omega(R)$-$\id_RM<\infty$. Then $\mathcal{B}_\omega(R)$-$\id_R\coOmega^n(M)<\infty$ by
\cite[Theorem 6.2]{HW} and \cite[Theorem 4.8]{H2}. If $\Ext_{R}^{\geqslant n+1}(\omega,M)=0$, then
$\coOmega^n(M)\in {_R\omega^{\bot}}$, and so $\coOmega^n(M)\in \mathcal{B}_\omega(R)$ by Lemma 4.1. It follows that $\mathcal{B}_\omega(R)$-$\id_RM\leqslant n$.

$(1)\Rightarrow (4)$ By \cite[Theorem 6.2]{HW}, $\mathcal{B}_\omega(R)$ is closed under extensions. By \cite[Proposition 3.7]{TH}, it is easy to see that
$\mathcal{P}_\omega(R)(=\Add_R\omega)$ is a $\mathcal{P}_\omega(R)$-proper generator
for $\mathcal{B}_\omega(R)$. Then the assertion follows from \cite[Theorem 3.7]{H2}.

$(4)\Rightarrow (5)$ Assume that exists an exact sequence:
$$0\rightarrow M\rightarrow X^M\rightarrow W^M\rightarrow 0$$ in $\Mod R$ such that $X^M\in \mathcal{B}_\omega(R)$ and $\mathcal{P}_\omega(R)$-$\id_RW^M\leqslant n-1$.
By \cite[Proposition 3.7]{TH}, there exists an exact sequence:
$$0\rightarrow X^{'}\rightarrow W_0\rightarrow X^M\rightarrow 0$$ in $\Mod R$ with $W_0\in \mathcal{P}_\omega(R)$ and $X'\in \mathcal{B}_\omega(R)$. Now consider the following pull-back diagram:
$$\xymatrix{ &   0 \ar[d] & 0 \ar[d]& & \\
&   X^{'} \ar@{=}[r] \ar[d]& X^{'}   \ar[d]& \\
0 \ar[r] &  W_M \ar[r] \ar[d]& W_0 \ar[r] \ar[d]& W^M \ar[r] \ar@{=}[d] & 0 \\
0 \ar[r] &  M \ar[r] \ar[d]& X^M \ar[r] \ar[d]& W^M \ar[r]  & 0 \\
& 0 & 0. & }$$
Then the leftmost column in the above diagram is the desired sequence.

$(5)\Rightarrow (4)$ Assume that there exists an exact sequence: $$0\rightarrow X_M\rightarrow W_M\rightarrow M\rightarrow 0$$ in $\Mod R$ such that
$X_M\in \mathcal{B}_\omega(R)$ and $\mathcal{P}_\omega(R)$-$\id_RW_M\leqslant n$. Then there exists an exact sequence:
$$0\rightarrow W_M\rightarrow W^0\rightarrow W^{'}\to 0$$ in $\Mod R$ with $W^0\in \mathcal{P}_\omega(R)$ and $\mathcal{P}_\omega(R)$-$\id_RW^{'}\leqslant n-1$.
Consider the following push-out diagram:
$$\xymatrix{
& & & 0 \ar[d] & 0 \ar[d] & \\
& 0 \ar[r] & X_M \ar@{=}[d] \ar[r] & W_M \ar[d]\ar[r] & M \ar[d]\ar[r] & 0\\
& 0 \ar[r] & X_M \ar[r] & W^0 \ar[d]\ar[r] & X \ar[d]\ar[r] & 0 \\
& & & W^{'} \ar[d]\ar@{=}[r] & W^{'} \ar[d] & \\
& & & 0 & 0. & }$$
It follows from \cite[Theorem 6.2]{HW} and the exactness of middle row in the above diagram that $X\in \mathcal{B}_\omega(R)$. So the rightmost column in the above
diagram  is the desired sequence.
\end{proof}

\begin{rem}\label{rem:4.3} {\rm The only place where the assumption that $\mathcal{B}_\omega(R)$-$\id_RM<\infty$ in Theorem 4.2 is used is in showing $(3)\Rightarrow (1)$.}
\end{rem}

If the given semidualizing module $_R\omega_S$
is faithful, then a module in $\Mod R$ with finite Bass injective dimension is in $\mathcal{B}_\omega(R)$ by \cite[Theorem 6.3]{HW}.
However, this property does not hold true in general.

\begin{exa} \label{exa: 4.4} {\rm Let $\Lambda$ be a finite-dimensional algebra
over an algebraically closed field given by the quiver:
\[\xymatrix{
1 \circ \ar[r] & \circ 2
  }\]
Put $\omega=I(1)\oplus I(2)$. Then $_{\Lambda}\omega_{\Lambda}$ is a semidualizing bimodule, but non-faithful since $\Hom_{\Lambda}(\omega,S(2))=0$.
We have an exact sequence $0\rightarrow S(2)\rightarrow I(2)\rightarrow I(1)\rightarrow 0$ in $\Mod \Lambda$. Both $I(1)$ and $I(2)$ are obviously
in $\mathcal{B}_\omega(\Lambda)$. But $S(2)$ is not in $\mathcal {B}_\omega(\Lambda)$ because $S(2)$ is not 2-$\omega$-cotorsionfree.}
\end{exa}

Motivated by \cite[Definition 2.4 and Lemma 2.5]{K}, we introduce the following

\begin{df} \label{df: 4.5} {\rm A semidualizing bimodule $_R\omega_S$ is called {\bf left} (resp. {\bf right}) {\bf semi-tilting} if
$\pd_R\omega<\infty$ (resp. $\pd_{S^{op}}\omega<\infty$).}
\end{df}

In the following, we will give an equivalent characterization of right semi-tilting bimodules in terms of the
finiteness of the Bass injective dimension of $_RR$.
We need the following two lemmas.

%Let $\mathcal{C}$ be the subclass of $\Mod R$ consisting of direct sums of copies of $C$.

\begin{lem} \label{lem: 4.6} Let $M\in \Mod R$ with $\mathcal{P}_\omega(R)$-$\id_RM\leqslant n(<\infty)$.
If $K\in \Mod R$ is isomorphic to a direct summand of $M$, then $\mathcal{P}_\omega(R)$-$\id_RK\leqslant n$.
\end{lem}

\begin{proof}
Note that $\mathcal{P}_\omega(R)=\Add_R\omega$ by Proposition 3.4(2). It is clear that $\mathcal{P}_\omega(R)\subseteq {^{\bot}\mathcal{P}_\omega(R)}$.
In addition, it is not difficult to verify that $\mathcal{P}_\omega(R)$ is $\mathcal{P}_\omega(R)$-coresolving in $\Mod R$ with
$\mathcal{P}_\omega(R)$ a $\mathcal{P}_\omega(R)$-coproper cogenerator in the sense of \cite{H2}.
Now the assertion follows from \cite[Corollary 4.9]{H2}.
\end{proof}

We use $\add_R\omega$ to denote the subclass of $\Mod R$ consisting of direct summands of finite direct sums of copies of $\omega$.

\begin{lem} \label{lem: 4.7} Let $M\in \Mod R$ be finitely generated and $n\geqslant 0$. If $\mathcal{P}_\omega(R)$-$\id_RM\leqslant n$,
then there exists an exact sequence:
$$0\to M \to \omega^0 \to \omega^1 \to \cdots \to \omega^n \to 0$$
in $\Mod R$ with all $\omega^i$ in $\add_R\omega$.
\end{lem}

\begin{proof} Let $\mathcal{P}_\omega(R)$-$\id_RM\leqslant n$ and
$$0\to M \stackrel{\alpha^{0}}{\longrightarrow} D^0 \stackrel{\alpha^{1}}{\longrightarrow} D^1
\stackrel{\alpha^{2}}{\longrightarrow} \cdots \stackrel{\alpha^{n}}{\longrightarrow} D^n \to 0\eqno{(4.1)}$$
be an exact sequence in $\Mod R$ with all $D^i$ in $\Add_R\omega$ ($=\mathcal{P}_\omega(R)$). Put $K^i=\Im \alpha^i$
for any $0\leqslant i \leqslant n$. There exists a module $G^0\in \Add_R\omega$ such that $D^0\oplus G^0$ is a
direct sum of copies of $\omega$, so we get a $\Hom_R(-,\mathcal{P}_\omega(R))$-exact exact sequence:
$$0\to M \stackrel{\beta^0}{\longrightarrow} D^0\oplus G^0 \stackrel{\beta^1}{\longrightarrow} D^1\oplus G^0
\stackrel{\beta^{2}}{\longrightarrow} D^2 \stackrel{\alpha^{3}}{\longrightarrow} \cdots \stackrel{\alpha^{n}}{\longrightarrow} D^n \to 0,$$
where $\beta^0={\binom {\alpha^0} {0}}$, $\beta^1={{\tiny \left(\begin{array}{cc} \alpha^1 & 0 \\ 0 & 1_{G^0}\end{array}\right)}}$
and $\beta^2=(\alpha^2, 0)$. Then $\Im \beta^1=K^1\oplus G^0$ and $\Im \beta^2=K^2$. Because $M$ is finitely generated by assumption,
there exist $\omega^0\in \add_R\omega$ and $H^0\in \Add_R\omega$ such that $D^0\oplus G^0=\omega^0\oplus H^0$ and $\Im \alpha^0\subseteq \omega^0$. So
we get an exact sequence:
$$0\to M \to \omega^0 \to L^0 \to 0\eqno{(4.2)}$$
in $\Mod R$ with $L^0\oplus H^0=\Im \beta^1$.

Consider the following push-out diagram with the middle row $\Hom_R(-,\mathcal{P}_\omega(R))$-exact exact
and the leftmost column splitting:
$$\xymatrix@C=20pt@R=15pt {
&0\ar[d]&0\ar[d]\\
& H^0\ar[d]\ar@{=}[r]& H^0\ar[d]\\
0\ar[r]&\Im \beta^1\ar[r]\ar[d]&D^1\oplus G^0\ar[d]\ar[r]& K^2\ar[r]\ar@{=}[d]&0\\
0\ar[r]& L^0\ar[d]\ar[r]&X^1\ar[d]\ar[r]& K^2 \ar[r]&0\\
&0&0.}\ \ \ $$
Then the middle column in the above diagram is $\Hom_R(-,\mathcal{P}_\omega(R))$-exact exact. From the proof of Lemma 4.6, we know that
$\Add_R\omega$ ($=\mathcal{P}_\omega(R)$) is $\mathcal{P}_\omega(R)$-coresolving in $\Mod R$. So $X^1\in \Add_R\omega$. Combining the exact sequences (4.1), (4.2)
with the bottom row in the above diagram, we get an exact sequence:
$$0\to M \to \omega^0 \to X^1 \to D^2 \stackrel{\alpha^{3}}{\longrightarrow} \cdots \stackrel{\alpha^{n}}{\longrightarrow} D^n \to 0$$
in $\Mod R$ with $\omega^0\in \add_R\omega$ and $X^1\in \Add_R\omega$. Repeating the above argument with $\Im (\omega^0\to X^1)$ replacing $M$, we
get an exact sequence:
$$0\to M \to \omega^0 \to \omega^1 \to X^2 \to D^3 \stackrel{\alpha^{4}}{\longrightarrow} \cdots \stackrel{\alpha^{n}}{\longrightarrow} D^n \to 0$$
in $\Mod R$ with $\omega^0,\omega^1\in \add_R\omega$ and $X^2\in \Add_R\omega$. Continuing this procedure, we finally get
an exact sequence:
$$0\to M \to \omega^0 \to \omega^1 \to \cdots \to \omega^n \to 0$$
in $\Mod R$ with all $\omega^i$ in $\add_R\omega$.
\end{proof}

We are now in a position to prove the following

\begin{thm} \label{thm: 4.8}
\begin{enumerate}
\item[]
\item If $_R\omega_S$ is right semi-tilting, then $\mathcal{B}_\omega(R)={_R\omega^{\perp}}$.
\item If $S$ is a left coherent ring, then $_R\omega_S$ is right semi-tilting with $\pd_{S^{op}}\omega\leqslant n$
if and only if $\mathcal{B}_\omega(R)$-$\id_RR\leqslant n$.
\end{enumerate}
\end{thm}

\begin{proof} (1) It follows from Corollary 3.2 and \cite[Theorem 3.9]{TH}.

(2) It is easy to see that $\mathcal{B}_\omega(R)$-$\id_RR\leqslant \mathcal{P}_\omega(R)$-$\id_RR=\pd_{S^{op}}\omega$. Now the necessity is clear.
Conversely, if $\mathcal{B}_\omega(R)$-$\id_RR=n<\infty$, then by Theorem 4.2, there exists a split exact sequence:
$$0\rightarrow X\rightarrow W\rightarrow R\rightarrow 0$$ in $\Mod R$ such that $X\in \mathcal{B}_\omega(R)$ and $\mathcal{P}_\omega(R)$-$\id_R{W}\leqslant n$.
So $W\cong X\oplus R$ and $\mathcal{P}_\omega(R)$-$\id_R{R}\leqslant n$ by Lemma 4.6. It follows from Lemma 4.7 that there exists an exact sequence:
$$0\to R \to \omega^0 \to \omega^1 \to \cdots \to \omega^n \to 0$$
in $\Mod R$ with all $\omega^i$ in $\add_R\omega$. Applying the functor $\Hom_R(-,\omega)$ to it we get the following exact
sequence:
$$0\to \Hom_R(\omega^n,\omega)\to \cdots \to \Hom_R(\omega^1,\omega) \to \Hom_R(\omega^0,\omega) \to \omega\to 0$$
in $\Mod S^{op}$ with all $\Hom_R(\omega^i,\omega)$ projective.
So $_R\omega_S$ is right semi-tilting with $\pd_{S^{op}}\omega\leqslant n$.
\end{proof}

Compare the following result with Lemma 3.9.

\begin{cor} \label{cor: 4.9} If $_R\omega_S$ is left and right semi-tilting, then
for every $M\in\Mod R$, $\mathcal{B}_\omega(R)$-$\id_RM=\sup\{i\geqslant 0\mid\Ext_R^i(\omega,M)\neq 0\}<\infty$.
\end{cor}

\begin{proof} Let $_R\omega_S$ be left and right semi-tilting. Then $\pd_R\omega<\infty$ and $\pd_{S^{op}}\omega<\infty$. Put
$\sup\{i\geqslant 0\mid\Ext_R^i(\omega,M)\neq 0\}=n$. Then $n<\infty$. It is easy to see that $_R\omega^{\bot}$-$\id_RM\geqslant n$.
So $\mathcal{B}_\omega(R)$-$\id_RM\geqslant n$ by Theorem 4.8(1).

We will use induction on $n$ to prove $\mathcal{B}_\omega(R)$-$\id_RM\leqslant n$.
If $n=0$, then $M\in{{_R\omega}^{\bot}}$. It follows from Theorem 4.8(1) that $M\in\mathcal{B}_\omega(R)$.
Now suppose $n\geqslant 1$. Then $\sup\{i\geqslant 0\mid\Ext_R^i(\omega,\coOmega^1(M))\neq 0\}=n-1$.
So $\mathcal{B}_\omega(R)$-$\id_R\coOmega^1(M)=n-1$ by the induction hypothesis, and hence
 $\mathcal{B}_\omega(R)$-$\id_RM\leqslant n$.
\end{proof}

\section {\bf The Bass Injective Dimension of Complexes}

In this section, we extend the Bass injective dimension of modules to that of complexes in derived categories. A {\bf cochain complex} $M^\bullet$
is a sequence of modules and morphisms in $\Mod R$ of the form:
$$\cdots \rightarrow M^{n-1}\stackrel{d^{n-1}}{\longrightarrow}M^n\stackrel{d^{n}}{\longrightarrow}M^{n+1}\rightarrow \cdots$$
such that $d^{n}d^{n-1} = 0$ for any $n\in \mathbb{Z}$, and the {\bf shifted complex} $M^\bullet[m]$ is the complex with
$M^{\bullet}[m]^n=M^{m+n}$ and $d^n_{M^{\bullet}[m]}=(-1)^md_{m+n}^n$. Any $M\in\Mod R$ can be considered as a complex
having $M$ in its 0-th spot and 0 in its other spots. We use $\mathbf{C}(R)$ and $\mathbf{D}^b(R)$
to denote the category of cochain complexes and the derived category of complexes with bounded
finite homologies of $\Mod R$, respectively. According to \cite[Appendix]{Lar}, the \textbf{supremum},
the \textbf{infimum} and the \textbf{amplitude} of a complex $M^\bullet$ are defined as follows:
$$\sup M^\bullet=\sup\{n\in \mathbb{Z}\mid H^n(M^\bullet) \neq 0\},$$
$$\inf M^\bullet=\inf\{n\in \mathbb{Z}\mid H^n(M^\bullet) \neq 0\},$$
$$\amp M^\bullet=\sup M^\bullet-\inf M^\bullet.$$

The Auslander category with respect to a dualizing complex was defined in \cite{Lar1}.
Dually we define the Bass class of complexes with respect to $\omega$ as follows.

\begin{df} \label{df: 5.1} {\rm A full subcategory $\mathcal{B}^{\bullet}_\omega(R)$ of $\mathbf{D}^b(R)$ consisting of complexes
$M^\bullet$ is called the {\bf Bass class} with respect to $\omega$ if the following conditions are satisfied:
\begin{enumerate}
\item ${\mathbf{R}\Hom_R(\omega,M^\bullet)}\in \mathbf{D}^b(R)$.
\item $\omega\otimes_S^{\mathbf{L}}{\mathbf{R}\Hom_R(\omega,M^\bullet)}\rightarrow M^\bullet$ is an isomorphism in $\mathbf{D}^b(R)$.
\end{enumerate}}
\end{df}

Let $M^\bullet\in \mathbf{C}(R)$ and $n\in \mathbb{Z}$. The \textbf{hard left-truncation} $\sqsubset^n$$M^\bullet$ of $M^\bullet$ at $n$ is given by:
$$\sqsubset^nM^\bullet:=\cdots \rightarrow 0\rightarrow 0\rightarrow M^n \stackrel{d^{n}}{\longrightarrow}M^{n+1}\stackrel{d^{n+1}}{\longrightarrow
}M^{n+2}\rightarrow\cdots.$$
Let $M^{\bullet}$ $\in \mathbf{D}^b(R)$ with $H(M^{\bullet})\neq 0$ and $\inf M^{\bullet}=i$. Taking an injective resolution $I^\bullet$ of $M^\bullet$. We define
the {\bf injective complex} $vI^\bullet=(\sqsubset^{i+1}$$I^\bullet)[1]$, which is unique up to an injective summand in degree $i$. In general,
we have that $H^t(vI^\bullet)\cong H^t(I^\bullet[1])$ if $t\geqslant i+1$. In particular, when $M^{\bullet}$ is a module $M$,
$vI^\bullet$ is isomorphic to $\coOmega^{1}(M)$ in $\mathbf{D}^b(R)$.

\begin{rem}\label{rem:5.2} {\rm (1) Let $M^\bullet\in \mathbf{D}^b(R)$. We see from the definition of $vI^\bullet$ that
there exists a distinguished triangle in $\mathbf{D}^b(R)$ of the form:
$$vI^\bullet[-1]\rightarrow M^{\bullet}\rightarrow I^i[-i]\rightarrow vI^\bullet.$$

(2) It is routine to check that $\mathcal{B}^{\bullet}_\omega(R)$ forms a triangulated subcategory of
$\mathbf{D}^b(R)$. Thus for an injective complex $I^\bullet$, $I^\bullet\in \mathcal{B}^{\bullet}_\omega(R)$
if and only if $vI^\bullet\in \mathcal{B}^{\bullet}_\omega(R)$.}
\end{rem}

\begin{lem} \label{lem: 5.3}  Let $M\in\Mod R$. Then the following statements are equivalent.
\begin{enumerate}
\item $\mathcal{B}_\omega(R)$-$\id_RM< \infty$.
\item $M\in \mathcal{B}^{\bullet}_\omega(R)$.
\end{enumerate}
\end{lem}

\begin{proof} (1) $\Rightarrow$ (2) Let $\mathcal{B}_\omega(R)$-$\id_RM< \infty$ and
$$0\rightarrow M\rightarrow Y^0\rightarrow Y^1\rightarrow \cdots \rightarrow Y^n\rightarrow 0$$
be an exact sequence in $\Mod R$ with all $Y^i$ in $\mathcal{B}_\omega(R)$. Then by Remark 5.2(2) and \cite[p.41, Corollary 7.22]{HJR},
we have $M\in \mathcal{B}^{\bullet}_\omega(R)$.

(2) $\Rightarrow$ (1) Let $M\in \mathcal{B}^{\bullet}_\omega(R)$ and $I^\bullet$ be an injective resolution of $M$.
Then $I^\bullet\in \mathcal{B}^{\bullet}_\omega(R)$ and $\mathbf{R}\Hom_R(\omega,M)\in \mathbf{D}^b(R)$. Put $s=\sup{\mathbf{R}\Hom_R(\omega,M)}$.
Because $H^i(\mathbf{R}\Hom_R(\omega,v^sI^\bullet))\cong H^{i+s}(\mathbf{R}\Hom_R(\omega,I^\bullet))=0$ for any $i\geqslant 1$.
It implies that $\coOmega^{s}(M)\in{_R\omega^{\bot}}$. By Remark 5.2(2) we have that $v^sI^\bullet\cong\coOmega^{s}(M)$ and
$v^sI^\bullet\in \mathcal{B}^{\bullet}_\omega(R)$, so $\coOmega^{s}(M)\in \mathcal{B}^{\bullet}_\omega(R)$, and hence
$\omega\otimes_S^{\mathbf{L}}\mathbf{R}\Hom_R(\omega,\coOmega^{s}(M))\rightarrow \coOmega^{s}(M)$ is an isomorphism in $\mathbf{D}^b(R)$,
equivalently $\omega\otimes_S{\coOmega^{s}(M)}_*\cong \coOmega^{s}(M)$ and $\Tor^S_{i\geqslant 1}(\omega,{\coOmega^{s}(M)}_*)=0$.
It follows that $\coOmega^{s}(M)\in \mathcal{B}_\omega(R)$ and $\mathcal{B}_\omega(R)$-$\id_RM\leqslant s$.
\end{proof}

We define the Bass injective dimension of complexes in $\mathbf{D}^b(R)$ as follows.

\begin{df} \label{df: 5.4} {\rm Let $M^\bullet$ be a complex in $\mathbf{D}^b(R)$. We define the
{\bf Bass injective dimension} of $M^{\bullet}$ as
\begin{displaymath}
\mathcal {B}^{\bullet}_\omega(R){\text -}\id M^{\bullet}:=\left\{ \begin{array}{ll}
\sup\mathbf{R}\Hom_R(\omega,M^{\bullet}) & \textrm{if $M^{\bullet}\in \mathcal {B}^{\bullet}_\omega(R)$,}\\
+\infty & \textrm{if $M^{\bullet}\notin \mathcal {B}^{\bullet}_\omega(R)$.}\\
\end{array} \right.
\end{displaymath}}
\end{df}

In the following result, we give an equivalent characterization when the Bass injective dimension of complexes is finite.

\begin{thm} \label{thm: 5.5}  Let $M^\bullet$ be a complex in $\mathbf{D}^b(R)$. Then the following statements are equivalent.
\begin{enumerate}
\item $\mathcal{B}^{\bullet}_\omega(R)$-$\id M^{\bullet}<\infty$.
\item There exists an isomorphism $M^{\bullet}\rightarrow Y^{\bullet}$ in $\mathbf{D}^b(R)$ with $Y^{\bullet}$
a bounded complex consisting of modules in $\mathcal{B}_\omega(R)$.
\end{enumerate}
\end{thm}
\begin{proof} $(2)\Rightarrow (1)$ The assertion follows from the fact that a complex $Y^{\bullet}$ of finite length consisting of modules in $\mathcal{B}_\omega(R)$
is in $\mathcal{B}^{\bullet}_\omega(R)$.

$(1)\Rightarrow (2)$ Let $\mathcal{B}^{\bullet}_\omega(R)$-$\id M^{\bullet}<\infty$. Then $M^{\bullet}\in \mathcal{B}^{\bullet}_\omega(R)$.
We will proceed by induction on $\amp M^\bullet$.
If $\amp M^\bullet$ $=0$, then there exists $T\in \Mod R$ such that $M^\bullet\cong T[-s]$, where $s=\sup M^\bullet$.
Since $\mathcal{B}_\omega(R)$-$\id_RT<\infty$ by Lemma 5.3, we have a quasi-isomorphism $T\rightarrow Y^{\bullet}$ with
$$Y^{\bullet}:=\cdots 0\rightarrow Y^0\rightarrow Y^1\rightarrow \cdots \rightarrow Y^n\rightarrow 0\rightarrow \cdots$$
a bounded complex and all $Y^i$ in $\mathcal{B}_\omega(R)$. Then the complex $Y^{\bullet}[-s]$ is the desired complex.

Now suppose $\amp M^\bullet\geqslant 1$. By Remark 5.2(1), there exists a distinguished triangle:
$$vI^\bullet[-1]\rightarrow M^{\bullet}\rightarrow I^i[-i]\stackrel{\alpha}{\longrightarrow} vI^\bullet$$ in $\mathbf{D}^b(R)$.
Since $\amp vI^\bullet<\amp M^\bullet$, by the induction hypothesis, there exists an isomorphism $\beta: vI^\bullet\rightarrow Y_1^{\bullet}$
in $\mathbf{D}^b(R)$ with $Y_1^{\bullet}$
a bounded complex consisting of modules in $\mathcal{B}_\omega(R)$. Thus we get another triangle:
$$vI^\bullet[-1]\rightarrow M^{\bullet}\rightarrow I^i[-i]\stackrel{\beta\alpha}{\longrightarrow} Y_1^\bullet$$ in $\mathbf{D}^b(R)$.
Furthermore, we have a triangle:
$$I^i[-i]\stackrel{\beta\alpha}{\rightarrow} Y_1^\bullet \rightarrow M^{\bullet}[1]\rightarrow I^i[-i+1]$$ in $\mathbf{D}^b(R)$.
Let $Y^{\bullet}_2$ be the mapping cone of $\beta\alpha$. Then there exists an isomorphism $M^{\bullet}[1]\rightarrow Y^{\bullet}_2$
in $\mathbf{D}^b(R)$. Put $Y^{\bullet}=Y^{\bullet}_2[-1]$. Then $Y^{\bullet}$ has finite length and all spots in $Y^{\bullet}$
are in $\mathcal{B}_\omega(R)$, and so $Y^{\bullet}$ is the desired complex.
\end{proof}

Let $\Lambda$ be an artin $R$-algebra over a
commutative artin ring $R$. We denote by $D$ the ordinary Matlis duality,
that is, $D(-):=\Hom_R(-,E^0(R/J(R)))$, where $J(R)$ is the
Jacobson radical of $R$ and $E^0(R/J(R))$ is the injective envelope
of $R/J(R)$. It is easy to verify that ($\Lambda, \Lambda$)-bimodule
$D(\Lambda)$ is semidualizing.
Recall that $\Lambda$ is called {\bf Gorenstein} if $\id_{\Lambda} \Lambda=\id_{\Lambda^{op}}\Lambda<\infty$. As an application of Theorem 5.5,
we get the following
%we provide an evident in the following result that the Bass injective dimension of modules is an effective tool in studying non-commutative rings.

\begin{cor} \label{cor: 5.6} Let $\Lambda$ be an artin algebra. Then the following statements are equivalent for any $n\geqslant 0$.
\begin{enumerate}
\item $\Lambda$ is Gorenstein with $\id_{\Lambda} \Lambda=\id_{\Lambda^{op}}\Lambda\leqslant n$.
\item For any simple module $T\in \Mod \Lambda$, $\mathcal{B}^{\bullet}_{D(\Lambda)}(\Lambda)$-$\id_{\Lambda}T\leqslant n$.
\item For any simple module $T\in \Mod \Lambda$, there exists a quasi-isomorphism $T\rightarrow Y^{\bullet}$ with $Y^{\bullet}$ a bounded complex of length at most $n+1$ consisting of modules in $\mathcal{B}_{D(\Lambda)}(\Lambda)$.
\item For any simple module $T\in \Mod \Lambda$, there exists an exact sequence:
$$0\rightarrow T\rightarrow X^T\rightarrow W^T\rightarrow 0$$ in $\Mod \Lambda$ such that $X^T\in \mathcal{B}_{D(\Lambda)}(\Lambda)$ and $\id_{\Lambda}W^T\leqslant n-1$.
\item For any simple module $T\in \Mod \Lambda$, there exists an exact sequence:
$$0\rightarrow X_T\rightarrow W_T\rightarrow T\rightarrow 0$$ in $\Mod \Lambda$ such that $X^T\in \mathcal{B}_{D(\Lambda)}(\Lambda)$ and $\id_{\Lambda}W_T\leqslant n$.
\end{enumerate}
\end{cor}

\begin{proof}
By Theorem 4.2, we have $(2)\Leftrightarrow (4)\Leftrightarrow (5)$. By Theorem 5.5, we have $(2)\Leftrightarrow (3)$.

$(1)\Rightarrow (2)$ Let $T\in \Mod \Lambda$ be simple. Since $\Lambda$ is Gorenstein with $\id_{\Lambda} \Lambda=\id_{\Lambda^{op}}\Lambda\leqslant n$,
it follows from \cite[Theorem 12.3.1]{EJ} that
$\coOmega^{n}(T)$ is Gorenstein injective. Then $\coOmega^{n}(T)\in\mathcal{B}_{D(\Lambda)}(\Lambda)$ by \cite[Corollary 5.2 and Theorem 3.9]{TH}.
Now the assertion follows from Lemma 5.3.

$(4)\Rightarrow (1)$ Let $T\in \Mod \Lambda$ be simple. Then by (4) and \cite[Theorem 3.9 and Corollary 4.2]{TH},
$\mathcal{GI}(\Lambda)$-$\id_{\Lambda}T\leqslant n$. So $\sup\{\mathcal{GP}(\Lambda)$-$\pd_{\Lambda}M\mid M\in \Mod \Lambda\}=
\sup\{\mathcal{GI}(\Lambda)$-$\id_{\Lambda}M\mid M\in \Mod \Lambda\}\leqslant n$ by \cite[Theorem 1.1]{BM} and \cite[Theorem 2.1]{Ta}.
It follows from \cite[Theorem 1.4]{HH} that $\Lambda$ is Gorenstein with $\id_{\Lambda} \Lambda=\id_{\Lambda^{op}}\Lambda\leqslant n$.
\end{proof}

\section {\bf A Dual of Auslander-Bridger's Approximation Theorem}

In this section, we first obtain a dual version of the Auslander-Bridger's approximation theorem, and then give its several
applications. We begin with the following

\begin{lem} \label{lem: 6.1} {\rm (\cite[Proposition 2.2]{W})}
\begin{enumerate}
\item For any $X\in \Mod R$, we have $(\theta_X)_*\cdot\mu_{X_*}=1_{X_*}$.
\item For any $Y\in \Mod S$, we have $\theta_{\omega\otimes_SY}\cdot(1_\omega\otimes\mu_Y)=1_{\omega\otimes_SY}$.
\end{enumerate}
\end{lem}

%\begin{proof} Let $f\in X_*$ and $c,c_1\in \omega$. We have $(\theta_X)_*\cdot\mu_{X_*}(f)(c)(c_1)
%=(\theta_X)_*(c\otimes f)(c_1)=\theta_X(c\otimes f)(c_1)=f(c)(c_1)$, so
%$(\theta_X)_*\cdot\mu_{X_*}=1_{X_*}$.
%\end{proof}

For any $n\geqslant 0$, recall from \cite{AR1} that the {\bf grade} of a finitely generated $R$-module $M$
is defined as $\grade_RM:=\inf\{i\geqslant 0\mid \Ext^i_R(M,R)\neq 0\}$; and the {\bf strong grade} of $M$,
denoted by $\sgrade_RM$, is said to be at least $n$ if $\grade_RX\geqslant n$ for any submodule $X$ of $M$.
We introduce two dual versions of these notions as follows.

\begin{df} \label{df: 6.2}
{\rm Let $M\in \Mod R$, $N\in \Mod S$ and $n\geqslant 0$.
\begin{enumerate}
\item The {\bf $\Ext$-cograde} of $M$ with respect to $\omega$ is defined as
$\Ecograde_\omega M:=\inf\{i\geqslant 0\mid\Ext_R^i(\omega,M)\neq 0\}$; and the {\bf strong $\Ext$-cograde} of $M$ with respect to $\omega$,
denoted by $\sEcograde_\omega M$, is said to be at least $n$ if $\Ecograde X\geqslant n$ for any quotient module $X$ of $M$.
\item The {\bf $\Tor$-cograde} of $N$ with respect to $\omega$ is defined as
$\Tcograde_\omega N:=\inf\{i\geqslant 0\mid\Tor^S_i(\omega,N)\neq 0\}$; and the {\bf strong $\Tor$-cograde} of $N$ with respect to $\omega$,
denoted by $\sTcograde_{\omega}N$, is said to be at least $n$ if $\Tcograde_\omega Y\geqslant n$ for any submodule $Y$ of $N$.
\end{enumerate}}
\end{df}

We remark that the $\Tor$-cograde of $N$ with respect to $\omega$ is called the {\bf cograde} of $N$ with respect to $\omega$ in \cite{TH}.

The following result can be regarded as a dual version of the Auslander-Bridger's approximation theorem
(see \cite[Proposition 3.8]{F}).

\begin{thm} \label{thm: 6.3} Let $M\in \Mod R$ and $n\geqslant 1$. If $\Tcograde_\omega\Ext^i_R(\omega,M)\geqslant i$ for any $1\leqslant i\leqslant n$,
then there exists a module $U\in \Mod R$ and a homomorphism $f: U\rightarrow M$ in $\Mod R$ satisfying the following properties:
\begin{enumerate}
\item $\mathcal{P}_\omega(R)$-$\id_RU\leqslant n$, and
\item $\Ext^i_R(\omega,f)$ is bijective for any $1\leqslant i\leqslant n$.
\end{enumerate}
\end{thm}

\begin{proof} We proceed by induction on $n$.

Let $n=1$ and $$Q_1\stackrel{f_1}{\longrightarrow} Q_0\rightarrow \Ext^1_R(\omega,M)\rightarrow 0$$
be a projective presentation of $\Ext^1_R(\omega,M)$ in $\Mod S$. Then we get the following exact sequence:
$$\omega\otimes_SQ_1\stackrel{1_\omega\otimes f_1}{\longrightarrow} \omega\otimes_SQ_0\rightarrow \omega\otimes_S\Ext^1_R(\omega,M)\rightarrow 0$$
in $\Mod R$ with both $\omega\otimes_SQ_1$ and $\omega\otimes_SQ_0$ in $\mathcal{P}_\omega(R)(=\Add_R\omega)$. Put $U=\Ker (1_\omega\otimes f_1)$.
Because $\omega\otimes_S\Ext^1_R(\omega,M)=0$ by assumption, $\mathcal{P}_\omega(R)$-$\id_RU\leqslant 1$.

Next we show that there exists a homomorphism $f: U\rightarrow M$ in $\Mod R$
such that $\Ext^1_R(\omega,f)$ is bijective. Since $Q_1$ and $Q_0$ are projective, there exist two homomorphisms $g_0$ and $g_1$
such that we have the following commutative diagram with exact rows:
$$\xymatrix{ Q_1 \ar@{.>}[d]^{g_1} \ar[r]^{f_1}& Q_0 \ar@{.>}[d]^{g_0} \ar[r]^{\delta^{'}\ \ \ \ \ \ }& \Ext^1_R(\omega,M) \ar@{=}[d]\ar[r]& 0\\
{I^0(M)}_* \ar[r]& {\coOmega^1(M)}_*\ar[r]^{\delta\ \ }& \Ext^1_R(\omega,M) \ar[r]& 0. \\
& {\rm Diagram\ (6.1)} & &}$$
Then there exists a homomorphism $f$
such that we have the following commutative diagram with exact rows:
$$\xymatrix{0\ar[r]&U \ar@{.>}[d]^{f} \ar[r]& \omega\otimes_SQ_1 \ar[d]^{h_1} \ar[r]^{1_\omega\otimes f_1}& \omega\otimes_SQ_0 \ar[d]^{h_0} \ar[r]& 0\\
0\ar[r]&M \ar[r]& I^0(M)\ar[r]& \coOmega^1(M) \ar[r]& 0,\\
& & {\rm Diagram\ (6.2)} & & }$$
where $h_1=\theta_{I^0(M)}\cdot (1_\omega\otimes g_1)$ and $h_0=\theta_{\coOmega^1(M)}\cdot (1_\omega\otimes g_0)$.
Applying the functor $(-)_*$ to Diagram (6.2), we obtain the following commutative diagram with exact rows:
$$\xymatrix{ (\omega\otimes_SQ_1)_* \ar[d]^{{h_1}_*} \ar[r]^{(1_\omega\otimes f_1)_*}& (\omega\otimes_SQ_0)_* \ar[d]^{{h_0}_*}
\ar[r]^{\delta^{''}}& \Ext^1_R(\omega,U) \ar[d]^{\Ext^1_R(\omega,f)} \ar[r]& 0\\
{I^0(M)}_*  \ar[r]& {\coOmega^1(M)}_*\ar[r]^{\delta\ \ }& \Ext^1_R(\omega,M)  \ar[r]& 0.\\
& {\rm Diagram\ (6.3)} & & }$$
Because the following diagram:
$$\xymatrix{ Q_0 \ar[d]^{\mu_{Q_0}} \ar[r]^{g_0 \ \ \ \ }& {\coOmega^1(M)}_* \ar[d]^{\mu_{{\coOmega^1(M)}_*}}\\
(\omega\otimes_SQ_0)_* \ar[r]^{(1_\omega\otimes g_0)_*\ \ \ \ }& (\omega\otimes_S{\coOmega^1(M)}_*)_*}$$
is commutative, $\mu_{{\coOmega^1(M)}_*}\cdot g_0=(1_\omega\otimes g_0)_*\cdot \mu_{Q_0}$. Then we have
\begin{align*}
&\ \ \ \ \ {h_0}_*\cdot \mu_{Q_0}\\
& =(\theta_{\coOmega^1(M)}\cdot (1_\omega\otimes g_0))_*\cdot \mu_{Q_0}\\
& =(\theta_{\coOmega^1(M)})_*\cdot (1_\omega\otimes g_0)_*\cdot \mu_{Q_0}\\
& =(\theta_{\coOmega^1(M)})_*\cdot\mu_{{\coOmega^1(M)}_*}\cdot g_0\\
& =1_{{\coOmega^1(M)}_*}\cdot g_0 \ \text{(by Lemma 6.1(1))}\\
& =g_0.
\end{align*}
On the other hand, from Diagrams (6.1) and (6.3), we get that $\delta^{'}=\delta\cdot g_0$ and
$\Ext^1_R(\omega,f)\cdot\delta^{''}=\delta\cdot {h_0}_*$. So we have
\begin{align*}
&\ \ \ \Ext^1_R(\omega,f)\cdot\delta^{''}\cdot\mu_{Q_0}\\
& =\delta\cdot {h_0}_*\cdot\mu_{Q_0}\\
& =\delta\cdot g_0\\
& =\delta^{'},
\end{align*}
and we get the following commutative diagram with exact rows:
$$\xymatrix{ (\omega\otimes_SQ_1)_* \ar[d]_{\cong}^{(\mu_{Q_1})^{-1}} \ar[r]^{(1_\omega\otimes f_1)_*}& (\omega\otimes_SQ_0)_* \ar[d]_{\cong}^{(\mu_{Q_0})^{-1}}
\ar[r]^{\delta^{''}}& \Ext^1_R(\omega,U) \ar[d]^{\Ext^1_R(\omega,f)} \ar[r]& 0\\
Q_1  \ar[r]^{f_1}& Q_0\ar[r]^{\delta^{'}\ \ \ \ \ }& \Ext^1_R(\omega,M)  \ar[r]& 0.}$$
Thus $\Ext^1_R(\omega,f)$ is bijective.

Now suppose $n\geqslant 2$. By the induction hypothesis, there exists a homomorphism $f^{'}: U^{'}\rightarrow M$ in $\Mod R$
such that $\mathcal{P}_\omega(R)$-$\id_RU^{'}\leqslant n-1$ and $\Ext^{i}_R(\omega,f^{'})$ is bijective for any $1\leqslant i \leqslant n-1$.
Then there exists a $\Hom_R(-, \mathcal{P}_\omega(R))$-exact exact sequence:
$$0\rightarrow U^{'}\stackrel{g^{'}}{\rightarrow} W \to X \to 0$$
in $\Mod R$ with $W$ in $\mathcal{P}_\omega(R)$, and we get
the following commutative diagram with exact columns and rows:
$$\xymatrix{
& & & 0 \ar[d] & 0 \ar@{.>}[d] & \\
& & & M \ar[d]^{\binom {1_M} {0}} \ar@{=}[r] & M \ar@{.>}[d] & \\
& 0 \ar[r] & U^{'} \ar@{=}[d] \ar[r]^{\binom {f^{'}} {g^{'}}} & M\oplus W \ar[d]^{(0,1_W)} \ar[r] & L \ar@{.>}[d] \ar[r] & 0 \\
& 0 \ar[r] & U^{'} \ar[r]^{g^{'}} & W \ar[d] \ar[r] & X \ar[r]\ar@{.>}[d] & 0 \\
& & & 0 & 0, & }$$
where $L=\Coker {\binom {f^{'}} {g^{'}}}$. It is easy to see that the exact sequence:
$$0\rightarrow U^{'} \stackrel{\binom {f^{'}} {g^{'}}}{\longrightarrow}M\oplus W\rightarrow L\rightarrow 0$$
is $\Hom_R(-, \mathcal{P}_\omega(R))$-exact. Because $\mathcal{P}_\omega(R)$-$\id_RU^{'}\leqslant n-1$ and $\Ext^{i}_R(\omega,f^{'})$ is bijective for any $1\leqslant i \leqslant n-1$,
we have that the sequence $$0\rightarrow {U^{'}}_*\stackrel{\binom {f^{'}} {g^{'}}_*\ }{\longrightarrow}
(M\oplus W)_*\rightarrow L_*\rightarrow 0$$ is exact, $\Ext_R^{1\leqslant i\leqslant n-1}(\omega,L)=0$
and $\Ext_R^{n}(\omega,M)\cong \Ext_R^{n}(\omega,L)$. Take a projective resolution:
$$Q_{n}\stackrel{f_n}{\longrightarrow} \cdots \stackrel{f_2}{\longrightarrow} Q_{1}\stackrel{f_1}{\longrightarrow} Q_{0}
\rightarrow \Ext^{n}_R(\omega,M)\rightarrow 0\eqno{(6.1)}$$
of $\Ext^{n}_R(\omega,M)$ in $\Mod S$. By assumption $\Tcograde_\omega\Ext^{n}_R(\omega,M)\geqslant n$, so we get the following exact sequence:
$$0\rightarrow N\rightarrow \omega\otimes_SQ_{n}\stackrel{1_\omega\otimes f_n}{\longrightarrow} \cdots \stackrel{1_\omega\otimes f_2}{\longrightarrow}
\omega\otimes_SQ_{1}\stackrel{1_\omega\otimes f_1}{\longrightarrow} \omega\otimes_SQ_{0}\rightarrow 0\eqno{(6.2)}$$
in $\Mod R$ with all $\omega\otimes_SQ_{i}$ in $\mathcal{P}_\omega(R)$ and $N=\Ker (1_\omega\otimes f_{n})$. Then $\mathcal{P}_\omega(R)$-$\id_RN\leqslant n$.
Applying the functor $(-)_*$ to the exact sequence (6.2) we get the following sequence:
$$0\rightarrow N_*\rightarrow (\omega\otimes_SQ_{n})_*\stackrel{(1_\omega\otimes f_n)_*}{\longrightarrow} \cdots
\stackrel{(1_\omega\otimes f_2)_*}{\longrightarrow} (\omega\otimes_SQ_{1})_*\stackrel{(1_\omega\otimes f_1)_*}{\longrightarrow} (\omega\otimes_SQ_{0})_*\rightarrow 0.\eqno{(6.3)}$$
Comparing the sequences (6.1) with (6.3) we get that $\Ext_R^{1\leqslant i \leqslant n-1}(\omega,N)=0$ and $\Ext_R^{n}(\omega,N)\cong \Ext_R^{n}(\omega,M)$.

Because $\Ext_R^{i}(\omega,L)=0$ for any $1\leqslant i \leqslant n-1$, we get an exact sequence:
$${I^0(L)}_*\rightarrow {I^1(L)}_*\rightarrow \cdots \rightarrow {I^{n-1}(L)}_*\rightarrow K_*\rightarrow \Ext_R^{n}(\omega,L)\rightarrow 0$$
in $\Mod S$, where $K=\Coker (I^{n-2}(L)\rightarrow I^{n-1}(L))$. Since all $Q_i$ are projective,
there exist homomorphisms $g_0, g_1,\cdots g_n$
such that we have the following commutative diagram with exact rows:
$$\xymatrix{ Q_{n} \ar@{.>}[d]^{g_n} \ar[r]^{f_n}&\cdots \ar[r]^{f_2}& Q_{1} \ar@{.>}[d]^{g_{1}} \ar[r]^{f_0} &Q_{0} \ar@{.>}[d]^{g_{0}} \ar[r]& \Ext^{n}_R(\omega,L) \ar@{=}[d]\ar[r]& 0\\
{I^0(L)}_*  \ar[r]&\cdots \ar[r]&{I^{n-1}(L)}_* \ar[r]& K_* \ar[r]& \Ext^{n}_R(\omega,L)  \ar[r]& 0. \\
& & {\rm Diagram\ (6.4)}& & & }$$
Then there exists a homomorphism $h$
such that we have the following commutative diagram with exact rows:
$$\xymatrix{0\ar[r]&N \ar@{.>}[d]^{h} \ar[r]^{s\ \ \ \ }& \omega\otimes_SQ_{n} \ar[d]^{h_{n}} \ar[r]^{1_\omega\otimes f_n}&\cdots
\ar[r]^{1_\omega\otimes f_1\ \ \ }& \omega\otimes_SQ_{1}\ar[d]^{h_{1}} \ar[r]^{1_\omega\otimes f_0}& \omega\otimes_SQ_{0} \ar[d]^{h_{0}} \ar[r]&0\\
0\ar[r]&L \ar[r]& I^0(L)\ar[r]&\cdots \ar[r]& I^{n-1}(L)\ar[r]&K  \ar[r]& 0, \\
& & & {\rm Diagram\ (6.5)} & & }$$
where $h_i=\theta_{I^{n-i}(L)}\cdot (1_\omega\otimes g_i)$ for any $1\leqslant  i\leqslant n$ and $h_0=\theta_{K}\cdot (1_\omega\otimes g_0)$.
Notice that the functor $(-)_*$ gives Diagram (6.5) back to Diagram (6.4), so $\Ext_R^{n}(\omega,h)$ is bijective.

Put $W^{'}=\omega\otimes_SQ_{n}$. Then we get an exact sequence:
$$0\rightarrow N\stackrel{\binom {h} {s}}{\longrightarrow} L\oplus W^{'}\rightarrow N^{'}\rightarrow 0$$
and a $\Hom_R(-, \mathcal{P}_\omega(R))$-exact exact sequence:
$$0\rightarrow U^{'} \stackrel{u}
{\longrightarrow} M\oplus W\oplus W^{'}\rightarrow L\oplus W^{'}\rightarrow 0$$
in $\Mod R$, where $u={\tiny \left(\begin{array}{c} f^{'} \\ g^{'} \\ 0\end{array}\right)}$.
Consider the following pull-back diagram:
$$\xymatrix{& & 0 \ar[d] & 0 \ar[d] & \\
0\ar[r] & U^{'}\ar[r]^{\alpha} \ar@{=}[d]& U \ar[r]^{\beta} \ar[d]^{\lambda}& N \ar[r] \ar[d]& 0\\
0\ar[r] & U^{'}\ar[r]^{u\ \ \ \ \ \ \ \ } & M\oplus W\oplus W^{'} \ar[r] \ar[d]& L\oplus W^{'} \ar[r] \ar[d]& 0\\
& & N^{'} \ar@{=}[r] \ar[d]&N^{'} \ar[d]\\
& & 0 & 0.& }$$
It is easy to see that the first row in the above diagram is $\Hom_R(-, \mathcal{P}_\omega(R))$-exact exact.
Because $\mathcal{P}_\omega(R)$-$\id_RU^{'}\leqslant n-1$ and $\mathcal{P}_\omega(R)$-$\id_RN\leqslant n$, $\mathcal{P}_\omega(R)$-$\id_RU\leqslant n$ by the dual version of \cite[Lemma 8.2.1]{EJ}.

Put $p=(1_M,0,0):M\oplus W\oplus W'\twoheadrightarrow M$ and $f=p\cdot \lambda$. Then $\Ext_R^i(\omega,f)=\Ext_R^i(\omega,p)\cdot\Ext_R^i(\omega,\lambda)$ for any $i\geqslant 0$.
Because $W\oplus W^{'}\in \mathcal{P}_\omega(R)$, $\Ext_R^i(\omega,p)$ is bijective for any $i\geqslant 1$.
Note that $\Ext^i_R(\omega,f^{'})$ is bijective for any $1\leqslant i\leqslant n-1$ and $\Ext_R^{1\leqslant i\leqslant n-1}(\omega,N)=0=\Ext_R^{1\leqslant i\leqslant n-1}(\omega,L)$.
We have the following commutative diagram with exact rows:
$$\xymatrix{&\Ext_R^{i}(\omega,U^{'}) \ar@{=}[d] \ar[r]^{\Ext_R^{i}(\omega,\alpha)}& \Ext_R^{i}(\omega,U) \ar[d]^{\Ext_R^{i}(\omega,\lambda)} \ar[r] &0 \\
0\ar[r]&\Ext_R^{i}(\omega,U^{'})\ar[r]^{\Ext_R^{i}(\omega,u)\ \ \ \ \ \ \ \ }&\Ext_R^{i}(\omega,M\oplus W\oplus W^{'})\ar[r]&0.
}$$
So $\Ext^i_R(\omega,\lambda)$ and $\Ext^i_R(\omega,f)$ are bijective for $1\leqslant i\leqslant n-1$.
On the other hand, because $\Ext_R^{n}(\omega,h)$ is bijective and $\Ext_R^{n+1}(\omega,U^{'})=0=\Ext_R^{n-1}(\omega,L)$, we have the following commutative diagram with exact rows:
$$\xymatrix{&\Ext_R^{n}(\omega,U^{'}) \ar@{=}[d] \ar[r]^{\Ext_R^{n}(\omega,\alpha)}& \Ext_R^{n}(\omega,U) \ar[d]^{\Ext_R^{n}(\omega,\lambda)} \ar[r]^{\Ext_R^{n}(\omega,\beta)}& \Ext_R^{n}(\omega,N) \ar[d]^{\cong}\ar[r]&0 \\
0\ar[r]&\Ext_R^{n}(\omega,U^{'})\ar[r]^{\Ext_R^{n}(\omega,u)\ \ \ \ \ \ \ \ }&\Ext_R^{n}(\omega,M\oplus W\oplus W^{'})\ar[r]&\Ext_R^{n}(\omega,L\oplus W^{'}).
}$$
So $\Ext^n_R(\omega,\lambda)$ and $\Ext_R^{n}(\omega,f)$ are bijective. The proof is finished.
\end{proof}

Dual to Theorem 6.3, we have the following

\begin{thm} \label{thm: 6.4} Let $N\in \Mod S$ and $n\geqslant 1$. If $\Ecograde_\omega\Tor^S_i(\omega,N)\geqslant i$ for any $1\leqslant i\leqslant n$,
then there exists a module $V\in \Mod S$ and a homomorphism $g: N\rightarrow V$ in $\Mod S$ satisfying the following properties:
\begin{enumerate}
\item $\mathcal {I}_\omega(S)$-$\pd_SV\leqslant n$, and
\item $\Tor_i^S(\omega,g)$ is bijective for any $1\leqslant i\leqslant n$.
\end{enumerate}
\end{thm}

In the rest of this section, we give several applications of Theorems 6.3 and 6.4.

Let $\Lambda$ be an artin $R$-algebra over a
commutative artin ring $R$ and $\mod \Lambda$ the class of
finitely generated left $\Lambda$-modules. It is well known that the ordinary Matlis duality functor $D(-)$ induces a duality between $\mod \Lambda$ and $\mod \Lambda^{op}$.
Recall from \cite{H} that $\Lambda$ is called \textbf{right quasi Auslander $n$-Gorenstein}
provided $\fd_{\Lambda^{op}}I^i(\Lambda_{\Lambda})\leqslant i+1$
for any $0\leqslant i\leqslant n-1$.
As an application of Theorem 6.3, we get the following

\begin{cor} \label{cor: 6.5} Let $\Lambda$ be a right quasi Auslander $n$-Gorenstein artin algebra and $M\in\mod \Lambda$.
%If $\grade\Ext^i_{\Lambda}(M,\Lambda)\geqslant i$ for any $1\leqslant i\leqslant n$,
Then there exists a module $U\in \mod \Lambda$ and a homomorphism $f: U\rightarrow M$ in $\mod \Lambda$ satisfying the following properties:
\begin{enumerate}
\item $\id_{\Lambda}U\leqslant n$, and
\item $\Ext^i_{\Lambda}(D(\Lambda),f)$ is bijective for any $1\leqslant i\leqslant n$.
\end{enumerate}
\end{cor}

\begin{proof} Let $M\in\mod \Lambda$ and $i,j\geqslant 0$. Then we have
\begin{align*}
&\ \ \  \Ext^i_{\Lambda}(D(\Lambda),M)\\
&\cong \Ext^i_{\Lambda}(D(\Lambda),D(D(M)))\\
& \cong D(\Tor_i^{\Lambda}(D(M),D(\Lambda))) \ \text{(by \cite[Chapter VI, Proposition 5.1]{CE})}\\
& \cong D(D(\Ext^i_{\Lambda^{op}}(D(M),\Lambda))) \ \text{(by \cite[Chapter VI, Proposition 5.3]{CE})}\\
& \cong \Ext^i_{\Lambda^{op}}(D(M),\Lambda).
\end{align*}
So for any $i\geqslant 1$ and $j\geqslant 0$, we have
\begin{align*}
&\ \ \ \Tor_j^{\Lambda}(D(\Lambda),\Ext^i_{\Lambda}(D(\Lambda),M))\\
&\cong \Tor_j^{\Lambda}(D(\Lambda),\Ext^i_{\Lambda^{op}}(D(M),\Lambda))\\
& \cong D(\Ext_{\Lambda}^j(\Ext_{\Lambda^{op}}^i(D(M),\Lambda), \Lambda) \ \text{(by \cite[Chapter VI, Proposition 5.3]{CE})}.
\end{align*}
Since $\Lambda$ is right quasi Auslander $n$-Gorenstein, $\grade_{\Lambda} \Ext_{\Lambda^{op}}^i(D(M),\Lambda)\geqslant i$ for any $1\leqslant i\leqslant n$
by \cite[Theorem 4.7]{AR1}.
It follows from the above argument that $\Tcograde_{D(\Lambda)}\Ext^i_{\Lambda}(D(\Lambda),M)\geqslant i$
for any $1\leqslant i\leqslant n$. In addition, notice that $D(\Lambda)$
is an injective cogenerator for $\Mod \Lambda$, so $\mathcal{P}_{D(\Lambda)}(\Lambda)$-$\id_{\Lambda}X=\id_{\Lambda}X$ for any $X\in \mod \Lambda$.
Now the assertion follows from Theorem 6.3.
\end{proof}

We give the second application of Theorems 6.3 (and 6.4) as follows.

\begin{cor} \label{cor: 6.6} Let $M\in \Mod R$ and $N\in \Mod S$. Then for any $n\geqslant 0$ we have
\begin{enumerate}
\item If $\Tcograde_\omega\Ext_R^i(\omega,M)\geqslant i+1$ for any $0\leqslant i\leqslant n$, then $\Ecograde_\omega M\geqslant n+1$.
\item If $\Ecograde_\omega\Tor^S_i(\omega,N)\geqslant i+1$ for any $0\leqslant i\leqslant n$, then $\Tcograde_CN\geqslant n+1$.
\end{enumerate}
\end{cor}
\begin{proof} (1) We proceed by induction on $n$.

Let $n=0$ and $\omega\otimes_SM_*=0$. Since $(\theta_M)_*\cdot \mu_{M_*}=1_{M_*}$ by Lemma 6.1(1), $\mu_{M_*}$ is a split monomorphism and $M_*=0$.

Now suppose $n\geqslant 1$. By the induction hypothesis, we have that $\Ecograde_\omega M\geqslant n$
and $\Ext_R^{0\leqslant i\leqslant n-1}(\omega,M)=0$. It is left to show $\Ext_R^n(\omega,M)=0$.
By Theorem 6.3, there exists a module $U\in \Mod R$ and a homomorphism $f: U\rightarrow M$ in $\Mod R$
such that $\mathcal{P}_\omega(R)$-$\id_RU\leqslant n$ and $\Ext^i_R(\omega,f)$ is bijective for any $1\leqslant i\leqslant n$.
It follows that $\Ext_R^{1\leqslant i\leqslant n-1}(\omega,U)=0$. Let
$$0\rightarrow U\stackrel{g}{\longrightarrow} W_0 \to W_1 \to \cdots  \to W_n\to 0$$
be an exact sequence in $\Mod R$ with all $W_i$ in $\mathcal{P}_\omega(R)$. Applying the functor $(-)_*$ to it we get an exact sequence:
$$0\to U_*\to {W_0}_*\to {W_1}_* \to\cdots \to {W_n}_*\to \Ext_R^n(\omega,U)\to 0$$
in $\Mod S$. Since $\Ext_R^n(\omega,M)\cong \Ext_R^n(\omega,U)$, we have $\Tcograde_\omega\Ext_R^n(\omega,U)\geqslant n+1$ by assumption.
Then we get the following commutative diagram with exact rows:
$$\xymatrix{ &
\omega\otimes_RU_* \ar[r] \ar[d]^{\theta_U}&
\omega\otimes_SW_0{_*} \ar[r] \ar[d]^{\theta_{W_0}}& \omega\otimes_SW_1{_*} \ar[r] \ar[d]^{\theta_{W_1}}
&\cdots \ar[r] & \omega\otimes_SW_n{_*} \ar[r] \ar[d]^{\theta_{W_n}}& 0 \\
0 \ar[r]&  U \ar[r] & W_0 \ar[r] &  W_1 \ar[r] &\cdots \ar[r]& W_n \ar[r]& 0.}$$
Because all $\theta_{W_i}$ are bijective, $\theta_{U}$ is epic.
%Consider the short exact sequence:
%$$0\rightarrow U\stackrel{\binom {f} {g}}{\longrightarrow} M\oplus W_0\rightarrow K\rightarrow 0.$$
%Because $\Ext_R^1(\omega,f)$ is bijective, we get an exact sequence:
%$$0\rightarrow U_*\rightarrow (M\oplus W_0)_*\rightarrow K_*\rightarrow 0.$$
Note that we have the following commutative diagram:
$$\xymatrix{\omega\otimes_SU_* \ar[r]^{1_\omega\otimes f_*\ \ } \ar[d]^{\theta_U}& \omega\otimes_SM_*\ar[d]^{\theta_{M}}\\
U \ar[r]^f & M.}$$
Because $\omega\otimes_SM_*=0$ by assumption. $f\cdot \theta_U=0$. But $\theta_U$ is epic, so $f=0$.
It follows that the bijection $\Ext_R^n(\omega,f)$ is zero and $\Ext_R^n(\omega,M)=0$.

(2) The proof is dual to that of (1), so we omit it.
\end{proof}

Before giving the third application of Theorem 6.3, we need the following

\begin{prop} \label{prop: 6.7} Let $$ V_1\buildrel {g} \over \longrightarrow V_0\to N \to 0\eqno{(6.4)}$$
be an exact sequence in $\Mod S$ satisfying the following conditions:
\begin{enumerate}
\item Both $\mu_{V_0}$  and $\mu_{V_1}$ are isomorphisms.
\item $\Ext^1_R(\omega,\omega\otimes_S{V_0})=0$ and $\Ext^1_R(\omega,\omega\otimes_S{V_1})=0=\Ext^2_R(\omega,\omega\otimes_S{V_1})$.
\end{enumerate}
Then there exists an exact sequence:
$$0\rightarrow \Ext_R^1(\omega,L)\rightarrow N\stackrel{\mu_{N}}{\longrightarrow}
(\omega\otimes_SN)_*\rightarrow\Ext_R^2(\omega,L)\rightarrow 0,$$ where $L=\Ker (1_\omega\otimes g)$.
\end{prop}

\begin{proof} By applying the functor $\omega\otimes_S-$ to (6.4), we get an exact sequence:
$$0\rightarrow L\rightarrow \omega\otimes_S{V_1}\stackrel{1_\omega\otimes g}{\longrightarrow}\omega\otimes_S{V_0}\rightarrow \omega\otimes_SN\rightarrow 0$$ in $\Mod R$.
Let $g=\alpha\cdot\pi$ (where $\pi:
V_1\twoheadrightarrow \Im g$ and $\alpha: \Im
g\rightarrowtail V_0$) and $1_\omega\otimes g=\alpha^{'}\cdot\pi^{'}$
(where $\pi^{'}: \omega\otimes_S{V_1}\twoheadrightarrow \Im (1_\omega\otimes g)$ and
$\alpha^{'}: \Im (1_\omega\otimes g)\rightarrowtail \omega\otimes_S{V_0}$) be the natural
epic-monic decompositions of $g$ and $1_\omega\otimes g$ respectively. Since $\Ext_R^1(\omega,\omega\otimes_S{V_0})=0$, we have the following commutative diagram with exact rows:
{\footnotesize $$\xymatrix{0 \ar[r]& \Im g \ar[r]^{\alpha} \ar@{-->}[d]^{h}& V_0 \ar[r] \ar[d]^{\mu_{V_0}}& N \ar[r] \ar[d]^{\mu_N}& 0\\
0 \ar[r]& (\Im (1_\omega\otimes g))_*\ar[r]^{{\alpha^{'}}_*} &(\omega\otimes_S{V_0})_* \ar[r]& (\omega\otimes_SN)_* \ar[r]& \Ext_R^1(\omega,\Im (1_\omega\otimes g)) \ar[r] &0,}$$}
where $h$ is an induced homomorphism. Then ${\alpha^{'}}_*\cdot h=\mu_{V_0} \cdot \alpha$. In addition, since $\mu_{V_0}$ is an isomorphism by assumption,
by the snake lemma we have $\Coker \mu_N\cong \Ext^1_R(\omega,\Im (1_\omega\otimes g))$ and $\Ker \mu_N\cong \Coker h$.

On the other hand, since $\Ext^1_R(\omega,\omega\otimes_S{V_1})=0=\Ext^2_R(\omega,\omega\otimes_S{V_1})$ by assumption, by applying the functor $(-)_*$ to the exact sequence:
$$0\rightarrow L\rightarrow \omega\otimes_S{V_1}\stackrel{\pi^{'}}{\rightarrow} \Im (1_\omega\otimes g)\rightarrow 0,$$ we get the following exact sequence:
$$0\rightarrow L_*\rightarrow
(\omega\otimes_S{V_1})_*\stackrel{{\pi^{'}}_*}{\longrightarrow}(\Im (1_\omega\otimes g))_{*}\rightarrow
\Ext_R^1(\omega,L)\rightarrow 0$$ and the isomorphism:
$$\Ext_R^1(\omega,\Im (1_\omega\otimes g)) \cong \Ext_R^2(\omega,L).$$ Because
$$\xymatrix{V_1 \ar[r]^{g}
\ar[d]^{\mu_{V_1}} & V_0 \ar[d]^{\mu_{V_0}} \\
(\omega\otimes_S{V_1})_* \ar[r]^{(1_\omega\otimes g)_*} & (\omega\otimes_S{V_0})_*}$$ is a commutative diagram, $(1_\omega\otimes g)_* \cdot \mu_{V_1}= \mu_{V_0} \cdot g$.
Because $1_\omega\otimes g=\alpha^{'}\cdot\pi^{'}$, $(1_\omega\otimes g)_*={\alpha^{'}}_*\cdot{\pi^{'}}_*$. Thus we have
${\alpha^{'}}_*\cdot h\cdot \pi=\mu_{V_0} \cdot \alpha\cdot \pi=\mu_{V_0} \cdot g=(1_\omega\otimes g)_* \cdot \mu_{V_1}=
{\alpha^{'}}_*\cdot{\pi^{'}}_*\cdot \mu_{V_1}$. Because ${\alpha^{'}}_*$ is monic, $h\cdot \pi={\pi^{'}}_*\cdot \mu_{V_1}$.
Notice that $\pi$ is epic and $\mu_{V_1}$ is an isomorphism, so $\Ker \mu_N\cong \Coker h \cong \Coker {\pi^{'}}_*\cong \Ext_R^1(\omega,L)$.
Consequently we obtain the desired exact sequence.
\end{proof}

As a consequence of Proposition 6.7, we have the following

\begin{cor} \label{cor: 6.8} Let $M\in \Mod R$. Then there exists an exact sequence:
$$0\rightarrow \Ext_R^1(\omega,M)\rightarrow \cTr_\omega M\stackrel{\mu_{\cTr_\omega M}}{\longrightarrow}
(\omega\otimes_S\cTr_\omega M)_*\rightarrow
 \Ext_R^2(\omega,M)\rightarrow 0.$$
\end{cor}

\begin{proof} Let $M\in \Mod R$. Then from the exact sequence (1.1) we get the following exact sequence:
$$0\to M_*\to {I^0(M)}_*\buildrel {{f^0}_*} \over \longrightarrow {I^1(M)}_* \to \cTr_\omega M \to 0$$
in $\Mod S$. Consider the following commutative diagram with exact rows:
{\footnotesize $$\xymatrix{0 \ar[r]& \Ker(1_\omega\otimes {f^0}_*) \ar[r] \ar@{-->}[d]^{h}
& \omega\otimes_S{I^0(M)}_* \ar[r]^{1_\omega\otimes {f^0}_*} \ar[d]^{\theta_{I^0(M)}}
& \omega\otimes_S{I^1(M)}_* \ar[r] \ar[d]^{\theta_{I^0(M)}}& \omega\otimes_S\cTr_\omega M \ar[r] & 0\\
0 \ar[r]& M\ar[r] & I^0(M) \ar[r]^{f^0} & I^1(M). & &}$$}
Because $I^0(M),I^1(M)\in \mathcal{B}_\omega(R)$ by \cite[Theorem 6.2]{HW},
both $\theta_{I^0(M)}$ and $\theta_{I^1(M)}$ are isomorphisms. So the induced
homomorphism $h$ is also an isomorphism and $M\cong \Ker(1_\omega\otimes {f^0}_*)$.
Note that ${I^0(M)}_*,{I^1(M)}_*\in \mathcal{A}_\omega(S)$ by \cite[Proposition 4.1]{HW}.
So both $\mu_{{I^0(M)}_*}$ and $\mu_{{I^1(M)}_*}$ are isomomorphisms, and then
the assertion follows from Proposition 6.7.
\end{proof}

We are now in a position to prove the following

\begin{thm} \label{thm: 6.9} For any $n\geqslant 1$, the following statements are equivalent.
\begin{enumerate}
\item $\sEcograde_\omega\Tor_i^S(\omega,N)$ $\geqslant i$ for any $N\in \Mod S$ and $1\leqslant i\leqslant n$.
\item $\sTcograde_\omega\Ext^i_R(\omega,M)$ $\geqslant i$ for any $M\in \Mod R$ and $1\leqslant i\leqslant n$.
\end{enumerate}
\end{thm}

\begin{proof} $(1)\Rightarrow (2)$ We proceed by induction on $n$.

Let $n=1$. Given a module $M$ in $\Mod R$, by Corollary 6.8 we have an exact sequence:
$$0\rightarrow \Ext_R^1(\omega,M)\rightarrow \cTr_\omega M\stackrel{\mu_{\cTr_\omega M}}{\longrightarrow} (\omega\otimes_S\cTr_\omega M)_*\rightarrow\Ext_R^2(\omega,M)\rightarrow 0.$$
Let $N=\Im\mu_{\cTr_\omega M}$ and let $\mu_{\cTr_\omega M}=\alpha\cdot\beta$ (where $\beta:
\cTr_\omega M\twoheadrightarrow N$ and $\alpha: N\rightarrowtail (\omega\otimes_S\cTr_\omega M)_*)$  be the natural
epic-monic decomposition of $\mu_{\cTr_\omega M}$.
Applying the functor $\omega\otimes_S-$ to the following exact sequence:
$$0\rightarrow \Ext_R^1(\omega,M)\rightarrow \cTr_\omega M\stackrel{\beta}{\longrightarrow} N\rightarrow 0,\eqno{(6.5)}$$
%$$0\rightarrow N\stackrel{\alpha}{\rightarrow}(\omega\otimes_S\cTr_\omega M)_*\rightarrow\Ext_R^2(\omega,M)\rightarrow 0.\eqno{(**)}$$
we get an exact sequence:
$$\Tor_1^S(\omega,N)\rightarrow \omega\otimes_S\Ext_R^1(\omega,M)\rightarrow \omega\otimes_S\cTr_\omega M\stackrel{1_\omega\otimes \beta}{\longrightarrow} \omega\otimes_SN\rightarrow 0.$$
Since $(1_\omega\otimes \alpha)\cdot (1_\omega\otimes \beta)=1_\omega\otimes \mu_{\cTr_\omega M}$ and $1_\omega\otimes \mu_{\cTr_\omega M}$ is a split monomorphism
by Lemma 6.1(2), $1_\omega\otimes \beta$ is an isomorphism. It follows that $\omega\otimes_S\Ext_R^1(\omega,M)$ is isomorphic to a quotient module
of $\Tor_1^S(\omega,N)$ in $\Mod R$. Then by assumption $\Ecograde_\omega(\omega\otimes_S\Ext_R^1(\omega,M))\geqslant 1$.
Using Corollary 6.6(2) we have that $\omega\otimes_S\Ext_R^1(\omega,M)=0$.

Let $X$ be a submodule of $\Ext_R^1(\omega,M)$ in $\Mod S$.
Then the exact sequence (6.5) induces the following two exact sequences:
$$0\rightarrow \Ext_R^1(\omega,M)/X\rightarrow (\cTr_\omega M)/X\stackrel{\gamma}{\longrightarrow} N\rightarrow 0,$$
$$0\rightarrow X \rightarrow\cTr_{\omega} M\stackrel{\pi}{\longrightarrow} (\cTr_\omega M)/X\to 0\eqno{(6.6)}$$
such that $\beta=\gamma \cdot \pi$. Then $1_\omega\otimes\beta=(1_\omega\otimes\gamma) \cdot (1_\omega\otimes\pi)$. On the other hand,
since $\omega\otimes_S\Ext_R^1(\omega,M)=0$, $\omega\otimes_S(\Ext_R^1(\omega,M)/X)=0$ and $1_\omega\otimes \gamma$ is bijective. So
$1_\omega\otimes \pi$ is also bijective. Hence from the exact sequence:
$$\Tor_1^S(\omega,(\cTr_\omega M)/X)\rightarrow \omega \otimes_SX\rightarrow
\omega \otimes_S\cTr_\omega M\stackrel{1_\omega \otimes \pi}{\longrightarrow} \omega \otimes_S(\cTr_\omega M)/X\rightarrow 0$$ induced by (6.6),
we get that $\omega \otimes_SX$ is isomorphic to a quotient module of
\linebreak
$\Tor_1^S(\omega ,(\cTr_\omega M)/X)$. Then by assumption $\Ecograde_\omega (\omega \otimes_SX)\geqslant 1$. It follows from Corollary 6.6(2)
that $\Tcograde_{\omega}X\geqslant 1$.

Now suppose $n\geqslant 2$. By the induction hypothesis, it suffices to prove that $\sTcograde_\omega \Ext^{n}_R(\omega ,M)\geqslant n$.
Because $\Ext^{n}_R(\omega ,M)\cong \Ext^{n-1}_R(\omega ,\coOmega^1(M))$,
\linebreak
$\sTcograde_\omega \Ext^{n}_R(\omega ,M)\geqslant n-1$ by the induction hypothesis.

Let $X$ be a submodule of $\Ext_R^n(\omega ,M)$ in $\Mod S$.
Because $\sTcograde_\omega \Ext^i_R(\omega ,M)$
\linebreak
$\geqslant i$ for any $1\leqslant i\leqslant n-1$, by Theorem 6.3 there exists a module
$U\in \Mod R$ and a homomorphism $f: U\rightarrow M$
in $\Mod R$ such that $\mathcal{P}_\omega (R)$-$\id_RU\leqslant n-1$ and $\Ext^i_R(\omega ,f)$ is bijective for any $1\leqslant i\leqslant n-1$.
Let $$0\rightarrow U\stackrel{g}{\longrightarrow} W_0 \to W_1 \to \cdots  \to W_{n-1}\to 0$$
be an exact sequence in $\Mod R$ with all $W_i$ in $\mathcal{P}_\omega (R)$ and
$L=\Coker{\binom {f} {g}}$. Then it is not difficult to verify that $\Ext^{1\leqslant i\leqslant n-1}_R(\omega ,L)=0$ and $\Ext^{n}_R(\omega ,M)\cong \Ext^{n}_R(\omega ,L)$.
%So we may consider $X$ as a submodule of $\Ext^{n+1}_R(\omega ,M')$. Consider a minimal injective resolution of $M'$ as follows:
%$$0\rightarrow M'\rightarrow I^0\rightarrow I^1\rightarrow \cdots \rightarrow I^n\rightarrow I^{n+1}. \eqno{(\dag)}$$
So we have an exact sequence:
$$0\rightarrow {L}_*\rightarrow {I^0(L)}_*\rightarrow {I^1(L)}_*\rightarrow \cdots \rightarrow {I^{n}(L)}_*\rightarrow Y\rightarrow 0$$
such that $\Ext_R^n(\omega ,L)\subseteq Y$. Applying the functor $\omega \otimes_S-$ to it, we get the following commutative diagram:
$$\xymatrix{ \omega \otimes_S{I^0(L)}{_*} \ar[r] \ar[d]_{\cong}^{\theta_{I^0(L)}}& \omega \otimes_S{I^1(L)}{_*} \ar[r] \ar[d]_{\cong}^{\theta_{I^1(L)}}
&\cdots \ar[r] & \omega \otimes_S{I^{n}(L)}{_*} \ar[r] \ar[d]_{\cong}^{\theta_{I^{n}(L)}}&  \omega \otimes_SY \ar[r] & 0 \\
I^0(L) \ar[r] & I^1(L) \ar[r] &\cdots \ar[r]& I^{n}(L)&.}$$
Because the bottom row in this diagram is exact, so is the upper row. It implies that $\Tor_{1\leqslant i \leqslant n-1}^S(\omega ,Y)=0$.
Since $X$ is isomorphic a submodule of $\Ext_R^n(\omega ,L)(\cong \Ext^{n}_R(\omega ,M))$ in $\Mod S$ and
$\sTcograde_\omega \Ext^{n}_R(\omega ,L)$=$\sTcograde_\omega \Ext^{n}_R(\omega ,M)\geqslant n-1$, $\Tcograde_CX\geqslant n-1$.
%Also because $\Ext^{n}_R(\omega ,M')\subseteq L$, there is an
%exact sequence $0\rightarrow X\rightarrow L\rightarrow L/X\rightarrow 0.$
Since $\Tor_{n-1}^S(\omega ,Y)=0$, we have an exact sequence:
$$\Tor_{n}^S(\omega ,Y/X)\rightarrow \Tor_{n-1}^S(\omega ,X)\rightarrow 0.$$
By assumption $\sEcograde_\omega \Tor_{n}^S(\omega ,Y/X)\geqslant n$, so
$\Ecograde_\omega \Tor_{n-1}^S(\omega ,X)\geqslant n$. Thus we have $\Ecograde_\omega \Tor_{i}^S(\omega ,X)\geqslant i+1$
for any $0\leqslant i\leqslant n-1$. It follows from Corollary 6.6(2) that $\Tcograde_CX\geqslant n$.

Dually, we get $(2)\Rightarrow (1)$.
\end{proof}

For any $n\geqslant 1$, recall that an artin algebra $\Lambda$ is called \textbf{Auslander $n$-Gorenstein}
provided $\fd_{\Lambda}I^i(_{\Lambda}\Lambda)\leqslant i$ for any $0\leqslant i\leqslant n-1$. The following result
extends \cite[Theorem 3.7]{F}.

\begin{cor} \label{cor: 6.10} Let $\Lambda$ be an artin algebra. Then the following statements are equivalent for any $n\geqslant 1$.
\begin{enumerate}
\item [$(1)\ \ \ $]$\Lambda$ is Auslander $n$-Gorenstein.
\item [$(1)^{op}$]$\Lambda^{op}$ is Auslander $n$-Gorenstein.
\item [$(2)\ \ \ $]$\sgrade_{\Lambda}\Ext^i_{\Lambda}(M, \Lambda)\geqslant i$ for any $M\in \mod \Lambda$ and $1\leqslant i\leqslant n$.
\item[$(2)^{op}$] $\sgrade_{\Lambda}\Ext^i_{\Lambda^{op}}(N, \Lambda)\geqslant i$ for any $N\in \mod \Lambda^{op}$ and $1\leqslant i\leqslant n$.
\item [$(3)\ \ \ $]$\sEcograde_{D(\Lambda)}\Tor_i^{\Lambda}(D(\Lambda),M)\geqslant i$ for any $M\in \mod \Lambda$ and $1\leqslant i\leqslant n$.
\item [$(4)\ \ \ $]$\sTcograde_{D(\Lambda)}\Ext^i_{\Lambda}(D(\Lambda),M)\geqslant i$ for any $M\in \mod \Lambda$ and $1\leqslant i\leqslant n$.
\end{enumerate}
\end{cor}

\begin{proof}
$(1)\Leftrightarrow (1)^{op}\Leftrightarrow (2)\Leftrightarrow (2)^{op}$ follow from \cite[Theorem 3.7]{F}.

Since the proof of Theorem 6.9 is also valid while modules are restricted to finitely generated modules over artin algebras,
$(3)\Leftrightarrow (4)$ holds true.

$(3)\Rightarrow (2)$ Let $M\in \mod \Lambda$ and $1\leqslant i\leqslant n$. If $Y$ is submodule of $\Ext^i_{\Lambda}(M,\Lambda)$ in $\mod \Lambda^{op}$,
then $D(Y)$ is isomorphic to a quotient module of $D(\Ext^i_{\Lambda}(M,\Lambda))$ in $\mod \Lambda$.
By \cite[Chapter VI, Proposition 5.3]{CE}, we have $D(\Ext^i_{\Lambda}(M,\Lambda))\cong \Tor_i^{\Lambda}(D(\Lambda),M)$.
So $\Ext^j_{\Lambda^{op}}(Y,\Lambda)\cong\Ext^j_{\Lambda}(D(\Lambda),D(Y))=0$ for any $0\leqslant j\leqslant i-1$ by (3).

(2) $\Rightarrow$ (3) Let $M\in \mod \Lambda$ and $1\leqslant i\leqslant n$. If $X$ is a quotient module of $\Tor_i^{\Lambda}(D(\Lambda),M)$ in $\mod \Lambda$,
then $D(X)$ is isomorphic to a submodule of $D(\Tor_i^{\Lambda}(D(\Lambda),M))$ in $\mod \Lambda^{op}$. By \cite[Chapter VI, Proposition 5.1]{CE}, we have
$D(\Tor_i^{\Lambda}(D(\Lambda),M))$ $\cong$ $\Ext^i_{\Lambda}(M, \Lambda)$. So $\Ext^j_{\Lambda}(D(\Lambda),X)\cong
\Ext^j_{\Lambda^{op}}(D(X),\Lambda)=0$ for any $0\leqslant j\leqslant i-1$ by (2).
\end{proof}

{\bf Acknowledgements.} This research was partially supported by NSFC (Grant Nos. 11171142, 11571164, 11501144),
a Project Funded by the Priority Academic Program Development of Jiangsu Higher Education
Institutions and NSF of Guangxi Province of China (Grant No. 2013GXNSFBA019005).
The authors thank the referee for the useful suggestions.

\end{document}